\documentclass[a4paper]{amsart}
\usepackage[margin=1in]{geometry}
\usepackage{amssymb}
\usepackage{graphicx}
\usepackage{epstopdf}
\usepackage{float}
\usepackage{import}
\usepackage{lipsum}
\usepackage{setspace}
\usepackage{hyperref}
\usepackage{color}
\usepackage{tikz}
\usepackage{tikz-cd}
\usepackage{amsmath}
\usepackage{amsthm}
\usepackage{enumerate}
\allowdisplaybreaks[4]
\newtheorem{Theorem}{Theorem}[section]
\newtheorem{Lemma}[Theorem]{Lemma}
\newtheorem{Definition}[Theorem]{Definition}
\newtheorem{Corollary}[Theorem]{Corollary}
\newtheorem{Proposition}[Theorem]{Proposition}
\newtheorem{Notation}[Theorem]{Notation}
\newtheorem{Remark}[Theorem]{Remark}
\newtheorem{Fact}[Theorem]{Fact}
\newtheorem{Example}[Theorem]{Example}

\begin{document}

\title{On the Number of Arrows of Cluster Quivers}

\author{Qiuning Du, \; Fang Li\; and \; Jie Pan}
\address{Department of Mathematics, Zhejiang University (Yuquan Campus), Hangzhou, Zhejiang 310027, P.R.China}
\email{11735007@zju.edu.cn, \; fangli@zju.edu.cn, \; panjie$\_$zhejiang@qq.com}
\date{version of \today}

\begin{abstract}
Let $\tilde{Q}$ (resp. $Q$) be an extended exchange (resp. exchange) cluster quiver of finite mutation type.
We introduce the distribution set of the number of arrows for $Mut[\tilde{Q}]$ (resp. $Mut[Q]$), give the maximum and minimum numbers of the distribution set and establish the existence of an extended complete walk (resp. a complete walk).
As a consequence, we prove that the distribution set for $Mut[\tilde{Q}]$ (resp. $Mut[Q]$) is continuous except the exceptional cases.

In case of cluster quivers $Q_{inf}$ of infinite mutation type, the number of arrows does not present a continuous distribution. Besides, we show that the maximal number of arrows of quivers in $Mut[Q_{inf}]$ is infinite if and only if the maximal number of arrows between any two vertices of a quiver in $Mut[Q_{inf}]$ is infinite.
\end{abstract}

\maketitle
\tableofcontents
\section{Introduction}
Cluster algebras were introduced by Fomin and Zelevinsky in \cite{FZ} with the original motivation to give a combinatorial characterization of dual canonical bases in Lie Theory and for the study of the total positivity. A significant notion in the theory of cluster algebras is cluster quiver. The algebraic property and combinatorial property are concerned in the study of cluster quivers, which are influential for the stucture of cluster algebras. In terms of algebraic property, cluster categories were generalized for cluster quivers, which are foundations of additive categorification of cluster algebras.  This paper is devoted to study on the combinatorial property of cluster quivers of finite mutatuion type and infinite mutation type.

Considering the combinatorial properties of quivers, there are two aspects, one is the number of vertices, the other is the number of arrows. Mutation preserves the number of vertices of cluster quivers, which is corresponding to the rank of cluster algebras. Thus we are interested in the variation of the number of arrows under mutations.

In \cite{S}, Ladkani classified the cluster quivers with the same number of arrows in their mutation equivalence class, and give representation-theoretic counterpart of mutation equivalence class with such combinatorial property: each class can be regarded as mutation equivalence class of quivers with potentials whose Jacobian algebras are all derived equivalent.
In \cite{BY}, Eric and Milen gave the calculation formula of the number of arrows for a quiver associated with a triangulation of a fixed surface and obtain the maximum number of arrows of the quiver arising from that surface.
Then a natural question is: how are the numbers of arrows distributed for the mutation equivalence class whose numbers of arrows are not constant?

We will mainly study the distribution of the number of arrows for cluster quivers in a mutation equivalence class.

We define an \textbf{extended exchange cluster quiver} $\tilde{Q}$ to be a cluster quiver whose arrows consist of arrows between mutable vertices and arrows  between a mutable vertex and a frozen vertex, and define  an \textbf{exchange cluster quiver} $Q$ to be a cluster quiver whose arrows consist of arrows between mutable vertices. Trivially, $Q$ can be seen as the full sub-quiver of $\tilde{Q}$ consisting of mutable vertices.

\emph{Because frozen variables we want to consider here come from sides on boundaries of surfaces, we will not consider to add frozen vertices for cluster quivers of exceptional type which do not arise from surfaces, that is, we are not concern with extended exchange cluster quivers of exceptional type.}

According to the calculation formulas of the numbers of arrows of $Q$ and $\tilde{Q}$ arising from surfaces (see Corollary \ref{calculation} and Proposition \ref{calculate-exQ}), the conditions for the numbers of arrows to be extremum in $Mut[Q]$ and $Mut[\tilde{Q}]$ are different, as well as so do the triangulations associated with quivers in $Mut[Q]$ and $Mut[\tilde{Q}]$ respectively. Therefore, we need to divide our discussion into two cases, namely with respect to $Q$ and $\tilde{Q}$ respectively. In particular, when $b=0$ for the surface, the conclusions for the two cases are the same.

We first give the maximum and minimum numbers of the distribution set (see Definition \ref{continuous}) for $Mut[Q]$ and $Mut[\tilde{Q}]$ respectively where $Q$ and $\tilde{Q}$ arise from a surface in the following proposition.
\begin{Proposition}[Proposition \ref{num-t} and Proposition \ref{num-check-t}]
  For $Q$ (resp. $\tilde{Q}$) arising from a surface $\mathbf{S}$, maximum number $t_{max}$ (resp. $\tilde{t}_{max}$) and minimum number $t_{min}$ (resp. $\tilde{t}_{min}$) of the distribution set for $Mut[Q]$ (resp. $Mut[\tilde{Q}]$) are given in Table \ref{type} (resp. Table \ref{type-t}).
\end{Proposition}
And we establish the existence of a complete walk and an extended complete walk (see Definition \ref{perfect}).
\begin{Lemma}[Lemma \ref{perfect1} and Lemma \ref{perfect2}]
There is a complete walk (resp. extended complete walk) in the graph $\mathbf{E^\circ(S,M)}$ of a surface $\mathbf{S}$ which is not the twice-punctured digon or the forth-punctured sphere (resp. the once-punctured triangle or the forth-punctured sphere).
\label{thmperfect}
\end{Lemma}
Then we obtain our main result (see Definition \ref{continuous} for a continuous distribution set ).
\begin{Theorem}[Theorem \ref{1} and Theorem \ref{2}]
Let $Q$ (resp. $\tilde{Q}$) be an exchange (resp. extended exchange) cluster quiver of finite mutation type, then the distribution set for $Mut[Q]$ (resp. $Mut[\tilde{Q}]$) is continuous except the exceptional cases (i) and (ii) (resp. (iii)) as follows:
\begin{enumerate}[(i)]
  \item $Q$ is of type $X_6$ or $X_7$;
  \item $Q$ arises from the twice-punctured digon or the forth-punctured sphere.
  \item $\tilde{Q}$ arises from the once-punctured triangle or the forth-punctured sphere.
\end{enumerate}
\label{main}
\end{Theorem}

The paper is organized as follows.

In Section 2, we provide necessary backgrounds of cluster quivers, mutation, triangulations of a surface and flips.

In Section 3, we first make preliminary discussion on the number of arrows for several surfaces(the twice-punctured monogon, the once-punctured digon, the twice-punctured digon and the forth-punctured sphere), then give the maximum and minimum numbers of the distribution set for $Mut[Q]$ where $Q$ arises from a surface, finally prove Theorem \ref{main} with respect to exchange cluster quivers.

In Section 4, we make a similar discussion on the distribution set for $Mut[\tilde{Q}]$, and give the proof of Theorem \ref{main} with respect to extended exchange cluster quivers.

In Section 5 we show the distribution set for any mutation equivalence class of infinite mutation type is not continuous. In fact, we have more information in this case as given in the following corollary finally:
\\
{\bf Corollary 5.7.}\;
  Let $Q$ be a connected quiver with $n\geqslant 3$ vertices, then the following are equivalent:
  \begin{itemize}
    \item $Q$ is mutation-infinite.
    \item There is no upper bound on the distribution set for $Mut[Q]$.
    \item There is a quiver in Mut$[Q]$ having lager than $2$ arrows between certain two vertices.
    \item For any two vertices of $Q$, the maximal number of arrows between this two vertices of a quiver in $Mut[Q]$ is infinite.
  \end{itemize}

\section{Preliminaries}
\subsection{Quivers of finite mutation type}
A quiver without loops or 2-cycles is said to be a \textbf{cluster quiver}. Let $Q_0$ be the set of vertices of $Q$, $Q_1$ be the set of arrows of $Q$, we denote by $t$ the number of arrows of a cluster quiver $Q$, which is $|Q_1|$.

Vertices in a cluster quiver are designated as \textbf{frozen} or \textbf{mutable}. We will always assume that there is no arrow between pairs of frozen vertices. The \textbf{mutation} of a cluster quiver $Q$ at the vertex $k$ is defined as follows:
\begin{itemize}
  \item Reverse all the arrows incident to $k$;
  \item If $Q$ has $m$ arrows $i \to k$, $n$ arrows $k \to j$ ,and $s$ arrows $j \to i$ (or $i \to j$), then there are $mn-s$ (or $mn+s$) arrows $i \to j$ (Turn arrows to the opposite direction if $mn-s$ is negative). Unless both $i$ and $j$ are frozen, in which case do nothing;
  \item Keep the other arrows of $Q$ unchanged.
\end{itemize}
\begin{Proposition}[\cite{FWZ}]
  (a) Let $k$ and $l$ be two mutable vertices which have no arrows between them. Then mutations at $k$ and $l$ commute with each other;
  (b) Mutation is an involution.
\end{Proposition}

Two cluster quivers $Q$ and $Q'$ are called \textbf{mutation-equivalent} if $Q$ can be obtained from $Q'$ by a sequence of mutations. The \textbf{mutation equivalence class} $Mut[Q]$ of a cluster quiver $Q$ is a set of cluster quivers (up to isomorphism) which are mutation equivalent to $Q$.

Let $a,b$ be positive integers satisfying $a\leq b$, write $[a,b]$ for $\{a,a+1,\dots,b\}$.

\begin{Definition}
 We call $W=\{t|t=|Q'_1|, where\ Q'\in Mut[Q]\}$ the \textbf{distribution set of the numbers of arrows} for ${Mut[Q]}$, briefly, the distribution set for ${Mut[Q]}$. Then the distribution set $W$ for ${Mut[Q]}$ is said to be \textbf{continuous} if $W=[a,b]$ where $a$ and $b$ are positive integers.
\label{continuous}
\end{Definition}
A cluster quiver $Q$ is said to be \textbf{of finite mutation type} if $Mut[Q]$ is finite. Felikson, Shapiro, and Tumarkin given a complete classification of cluster quivers of finite mutation type.

\begin{Theorem}\cite{FST1}
A cluster quiver $Q$ with $n$ vertices is of finite mutation type if and only if:
\begin{enumerate}[(a)]
  \item $n\leq2$
  \item $Q$ arises from a 2-dimensional Riemann surface$(n\geq3)$
  \item $Q$ is mutation equivalent to one of the $11$ exceptional types (see Figure \ref{exceptional}):
  $$E_6,\ E_7,\ E_8,\ \widetilde{E}_6,\ \widetilde{E}_7,\ \widetilde{E}_8,\ E_6^{(1,1)},\ E_7^{(1,1)},\ E_8^{(1,1)},\ X_6,\ X_7$$
\end{enumerate}
\begin{figure}[H]
\centering
\includegraphics[width=12.5cm]{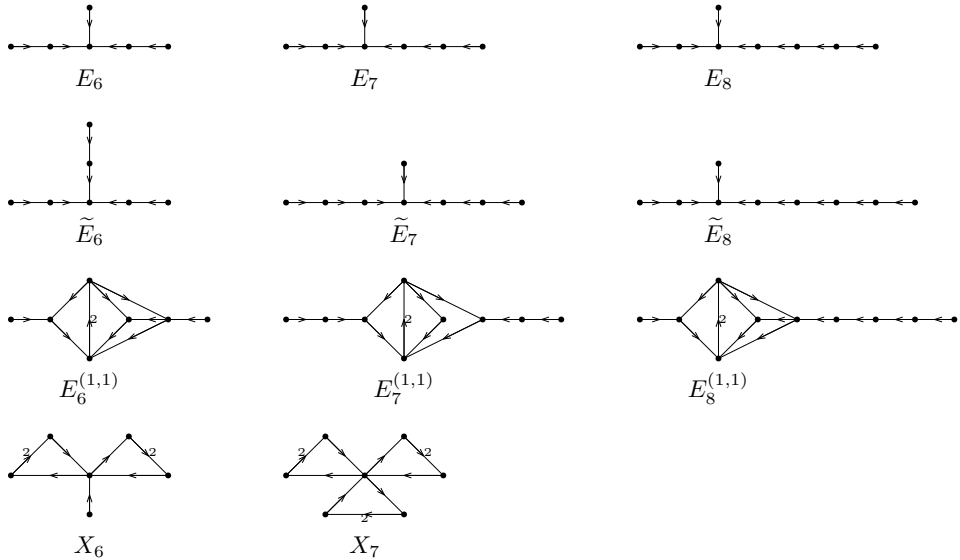}
\caption{Eleven exceptional types}
\label{exceptional}
\end{figure}
\end{Theorem}

Thanks to the classification in this theorem, we will discuss the distribution set for $Mut[Q]$ and $Mut[\tilde{Q}]$ respectively due to  the definitions of exchange cluster quivers $Q$ and extended exchange cluster quiver $\tilde{Q}$ in Section 1.

\subsection{Cluster quivers arising from surfaces}
We now recall how extended exchange cluster quivers and exchange cluster quivers can be associated with triangulations of surfaces. Because the distribution set of the numbers of arrows are the same for the ordinary triangulated surface and the tagged triangulated surface, we only consider the ordinary triangulations for simplicity. We mainly follow the works of Fomin, Shapiro and Thurston \cite{FS}.

A \textbf{marked surface} is a pair $(\mathbf{S},\mathbf{M})$ ($\mathbf{S}$ for abbreviation), where $\mathbf{S}$ is a connected oriented 2-dimensional Riemann surface and $\mathbf{M}$ is a finite set of marked points in the closure of $\mathbf{S}$. Marked points in the interior of $\mathbf{S}$ are called \textbf{punctures}, marked points on the boundary of $\mathbf{S}$ are called \textbf{boundary points}. Each boundary component contains at least one boundary points. The \textbf{degree} of a marked point is defined as the sum of the numbers of different edges in the underlying graph which incident with the marked point.

In this paper, we require the same as that in \cite{FS} that $\mathbf{S}$ is none of the following:
\begin{itemize}
  \item a sphere with one, two or three punctures;
  \item an un-punctured or once-punctured monogon;
  \item an un-punctured digon;
  \item an un-punctured triangle.
\end{itemize}

An \textbf{arc} $\gamma$ in $\mathbf{S}$ is a curve, considered up to isotopy, such that:
\begin{itemize}
  \item the endpoints of $\gamma$ are in M;
  \item $\gamma$ does not cross itself, except that its endpoints may coincide;
  \item except the endpoints, $\gamma$ is disjoint from the boundary of S;
  \item $\gamma$ does not cut out a monogon or a digon.
\end{itemize}

Two arcs are called \textbf{compatible} if they do not intersect in the interior of $\mathbf{S}$. A \textbf{triangulation} $T$ of $\mathbf{S}$ is a maximal collection of distinct pairwise compatible arcs. Given a triangulated surface, the number $n$ of arcs can be calculated by the following formula.

\begin{Proposition}\cite{FG}
Each triangulation consists of
$$n=6g+3b+3p+c-6$$
arcs, where $g$ is the genus of $\mathbf{S}$, $b$ is the number of boundary components, $p$ is the number of punctures, and $c$ is the number of boundary points.
\end{Proposition}

The arcs of a triangulation separate the surface into \textbf{ideal triangles}. We allow \textbf{self-folded} triangles (see in Figure.\ref{self-folded}). We also allow for a possibility that two triangles share more than one side.
\begin{figure}[H]
\centering
\includegraphics[height=1.5cm]{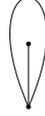}
\caption{a self-folded triangle}
\label{self-folded}
\end{figure}

Let $T$ be a triangulation of a surface $\mathbf{S}$, an extended exchange cluster quiver is said to be \textbf{associated with} $T$ denoted as $\tilde{Q}_T$ as follows:
\begin{itemize}
  \item mutable vertices correspond to arcs and frozen vertices correspond to sides on the boundary;
  \item if two arcs, or an arc and a side on the boundary, belongs to the same triangle, there is an arrow between the corresponding vertices in the clockwise of the boundary of the triangle. Note that there is no arrows between two frozen vertices;
  \item for an arc $p$ which is the one in the loop $q$ of a self-fold triangle, there is an arrow $i \to p$ (or $p \to j$) in $\tilde{Q}_T$ when there exists an arrow $i \to q$ (or $q \to j$);
  \item if there are two triangles sharing two sides in $T$, delete the 2-cycle in $\tilde{Q}_T$ of which arrows are between the corresponding vertices of the common sides of triangles.
\end{itemize}
And an exchange cluster quiver associated with $T$ denoted by $Q_T$ is the full sub-quiver of $\tilde{Q}_T$ consisting of mutable vertices.
Then the associated extended exchange (resp. exchange) cluster quiver $\tilde{Q}_T$ (resp. $Q_T$) is said to \textbf{arise from} a surface $\mathbf{S}$. Let $\tilde{t}$ (resp. $t$) be the number of arrows of $\tilde{Q}_T$ (resp. $Q_T$), then the triangulation $T$ is said to be \textbf{equipped with} $\tilde{t}$ (resp. $t$) arrows.

A \textbf{flip} is a transformation of a triangulation $T$ by replacing an arc with a unique different arc and leaving the other arcs unchanged to create a new triangulation $T'$.
\begin{Proposition}\cite{FS}
Let $\tilde{Q}_T$ and $\tilde{Q}_{T'}$ be two extended exchange cluster quivers associated with $T$ and $T'$ respectively. If $T'$ is obtained from $T$ by a flip at the arc $k$, then $\tilde{Q}_{T'}$ is obtained from $\tilde{Q}_T$ by a mutation at the vertex $k$.
\label{f-m}
\end{Proposition}
\begin{Proposition}\cite{FS}
For an extended exchange cluster quiver $\tilde{Q}_T$ arising from a surface $\mathbf{S}$, the mutation equivalence class $Mut[\tilde{Q}_T]$ depends only on $\mathbf{S}$ but not on the choice of an ideal triangulation.
\label{prop-FS}
\end{Proposition}
Obviously, the propositions above hold for exchange cluster quivers.
Proposition \ref{prop-FS} shows that the mutation equivalence class can be described by the triangulated surface.
Then the distribution set $W$ for $Mut[\tilde{Q}_T]$ (resp. $Mut[Q_T]$) where $\tilde{Q}_T$ (resp. $Q_T$) arises from a surface $\mathbf{S}$ is denoted as $\tilde{W}_\mathbf{S}$ (resp. $W_\mathbf{S}$) and called the distribution set for $\mathbf{S}$ with respect to $\tilde{Q}$ (resp. $Q$) for short.

Let $\mathbf{E^\circ(S,M)}$ denote a graph of which vertices are ideal triangulations of $\mathbf{(S,M)}$, and edges correspond to flips.
\begin{Proposition}\cite{FS}
The graph $\mathbf{E^\circ(S,M)}$ is connected.
\end{Proposition}

According to Remark 4.2 in \cite{FS}, an arbitrary ideal triangulation can be obtained by gluing a number of "puzzle pieces" together (shown below) in a specific way:
\begin{figure}[H]
\centering
\includegraphics[height=2cm]{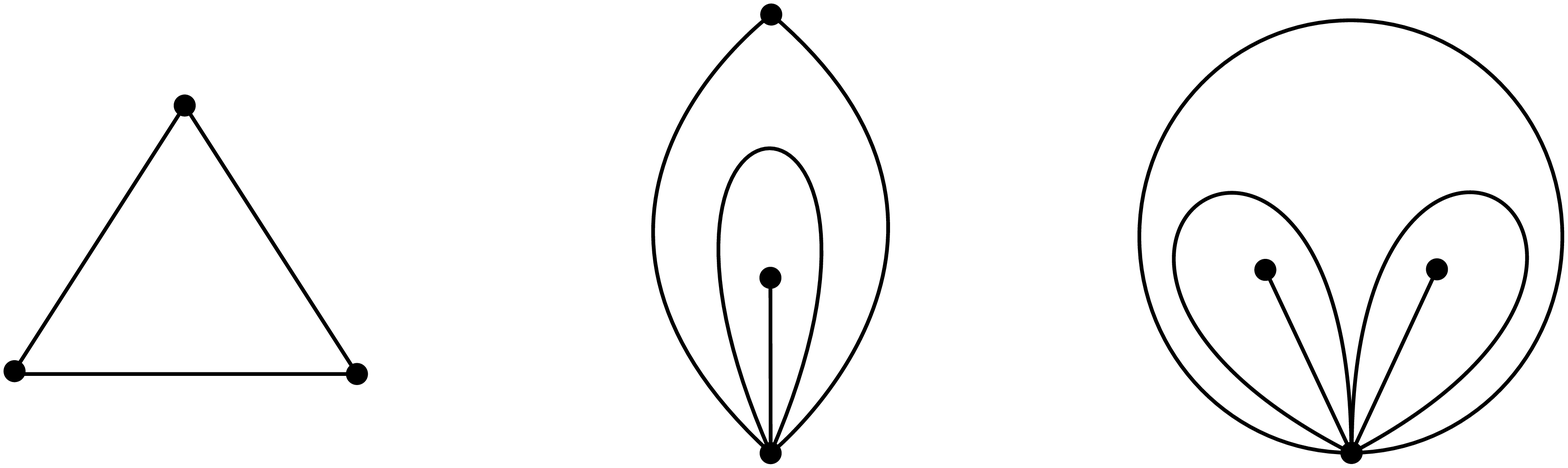}
\label{3puzzle}
\end{figure}
\begin{itemize}
  \item Take several puzzle pieces;
  \item Take a partial matching of the exposed (exterior) sides of these puzzle pieces, never matching two sides of the same puzzle piece. In order to obtain a connected surface, we need to ensure that any two puzzle pieces in the collection are connected via matched pairs;
  \item Glue the puzzle pieces along the matched sides, making sure that the orientations match.
\end{itemize}

While a triangulation of the forth-punctured sphere can not be obtained by this construction, which is obtained by gluing three self-folded triangles to respective sides of an ordinary triangle.

In order to analyse the distribution of the numbers of arrows, we split those three puzzle pieces into nine specific pieces shown in Figure \ref{ori+tri} by defining the side as the \textbf{matched side}(labeled by $\circ$) or the \textbf{frozen side}(labeled by $\square$), in which the dotted line represents the sphere.
The frozen side corresponds to the side on the boundary, which can not be glued.
The matched side corresponds to the common side of two ideal triangles.
The unlabeled sides in the interior of the piece $\triangle_i\ (i=4,\cdots,9)$ corresponds to arcs in the triangulation.

Then any triangulations can be obtained from nine pieces.
\begin{figure}[H]
\centering
\includegraphics[width=12cm]{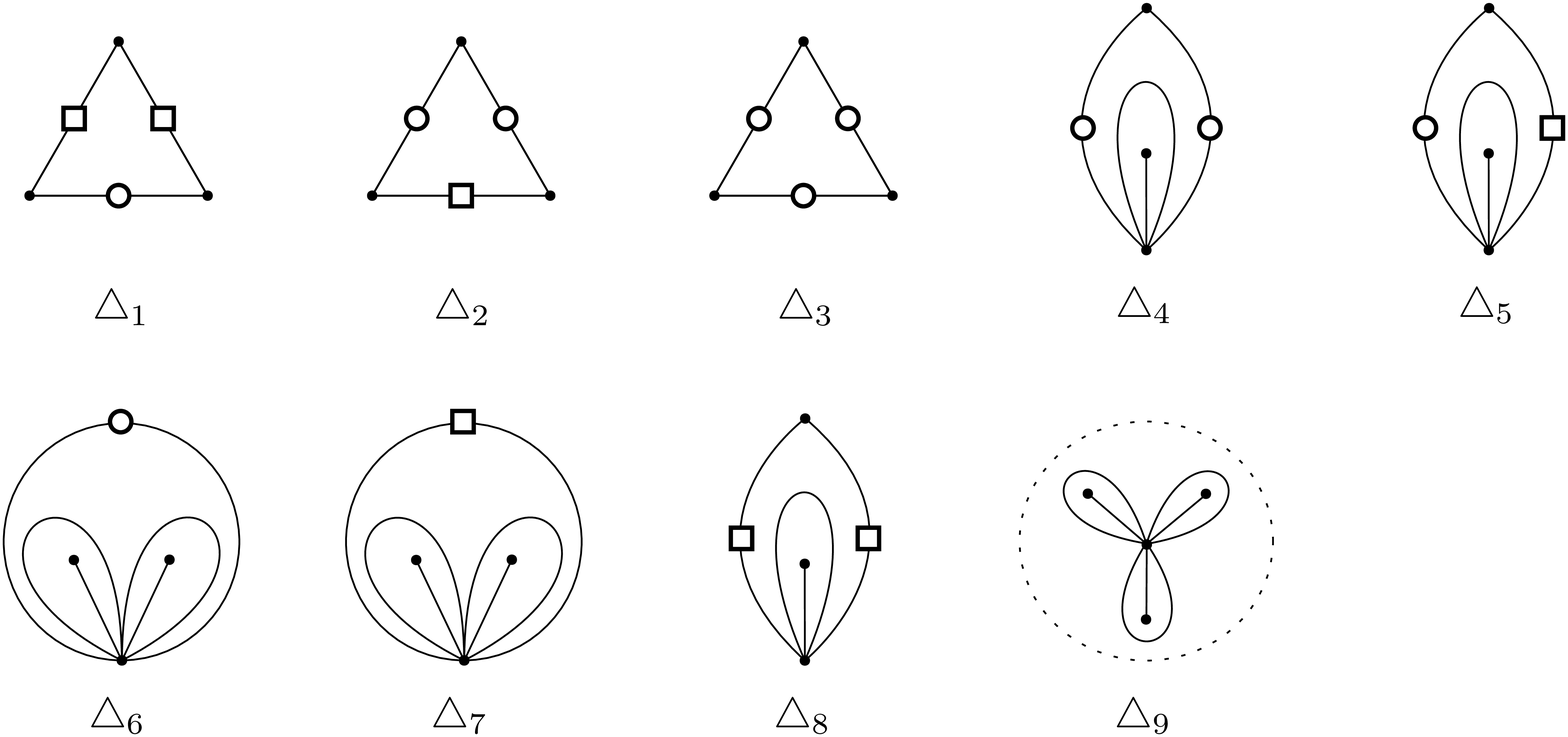}
\caption{Nine pieces}
\label{ori+tri}
\end{figure}

\section{The proof of Theorem \ref{main} with respect to exchange cluster quivers}
In this section, we only concern about exchange cluster quivers.
\subsection{Two lemmas}
We first discuss the distribution sets for several special surfaces. By the enumeration method, we have the following lemma.

\begin{Lemma}
The distribution sets for $\mathbf{S_{2p-mon}}$ and $\mathbf{S_{2p-mon}}$ with respect to $Q$ are continuous, where $\mathbf{S_{2p-mon}}$ is the twice-punctured monogon and $\mathbf{S_{1p-dig}}$ is the once-punctured digon.
\label{78}
\end{Lemma}
\begin{proof}
For $\mathbf{S_{2p-mon}}$, there are triangulations equipped with the same number of arrows. For each number, we list one triangulation in Figure \ref{2p-mon}. Then the distribution set for $\mathbf{S_{2p-mon}}$ is $\{4,5,6\}$.

It is easy to see that the distribution set $W_\mathbf{S_{1p-dig}}$ is $\{0\}$.
\begin{figure}
\centering
\includegraphics[width=8cm]{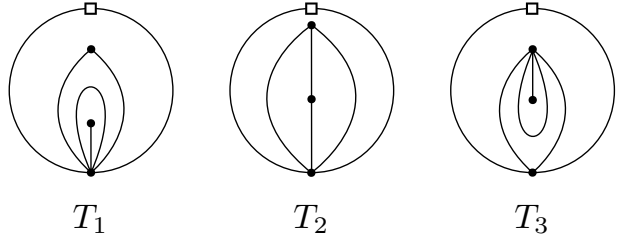}
\caption{Triangulations of $\mathbf{S_{2p-mon}}$}
\label{2p-mon}
\end{figure}
\end{proof}
\begin{Lemma}
The distribution sets for $\mathbf{S_{2p-dig}}$ and $\mathbf{S_{4p-sphere}}$ with respect to $Q$ are not continuous, where $\mathbf{S_{2p-dig}}$ is the twice-punctured digon and $\mathbf{S_{4p-sphere}}$ is the forth-punctured sphere. More specifically,
$$W_\mathbf{S_{2p-dig}}=\{4,6,7,8\},\ W_\mathbf{S_{4p-sphere}}=\{8,9,10,12\}.$$
\label{bulianxu}
\end{Lemma}
\begin{proof}
For $\mathbf{S_{2p-dig}}$, any triangulation is isotopy to one of the graphs shown in the left of Table \ref{2p+4p}. There is no exchange cluster quiver with 5 arrows arising from $\mathbf{S_{2p-dig}}$.

For $\mathbf{S_{4p-sphere}}$, any triangulation is isotopy to one of graphs shown in the right of Table \ref{2p+4p}. There is no exchange cluster quiver with 11 arrows arising from $\mathbf{S_{4p-sphere}}$.
\begin{table}
\centering
\includegraphics[width=16cm]{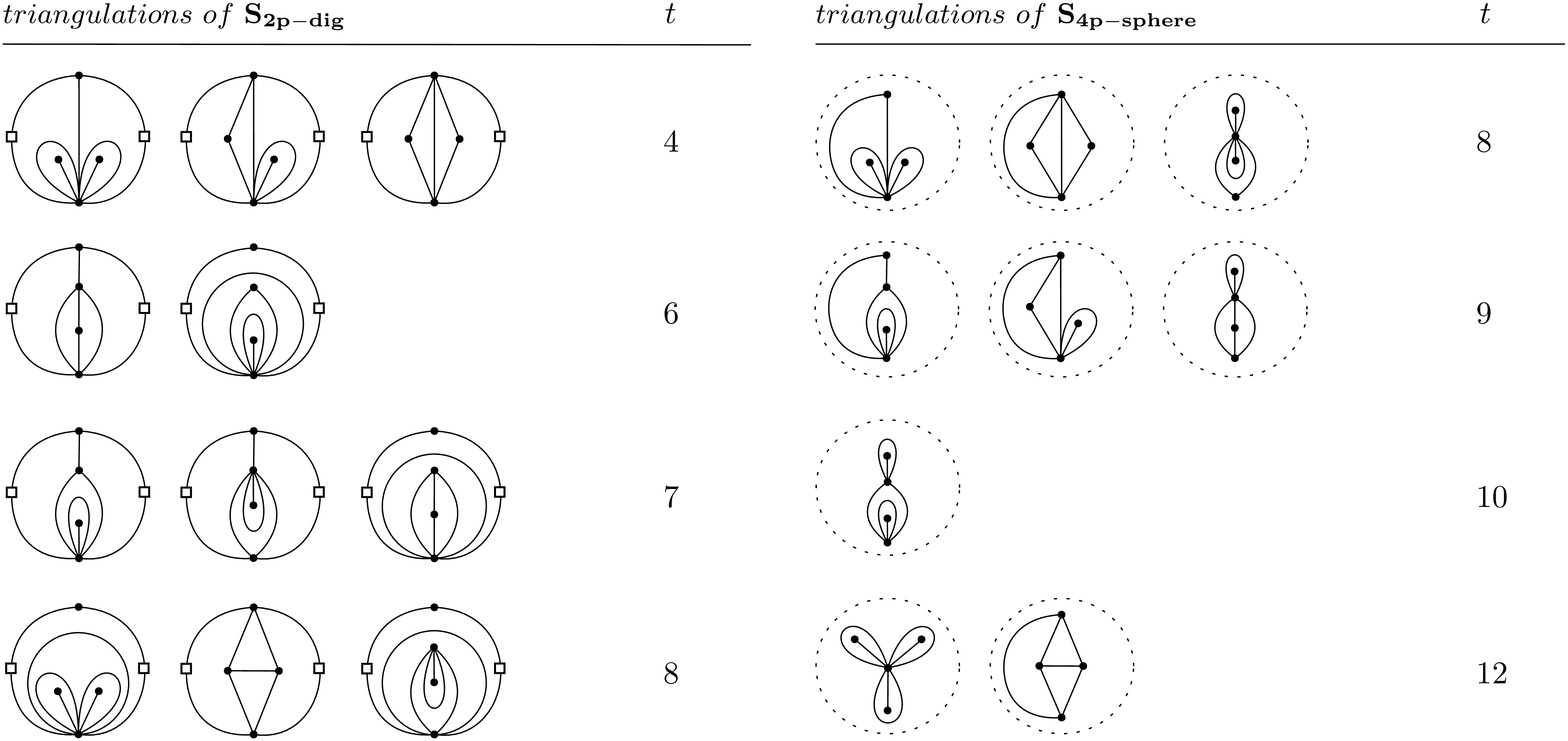}
\caption{Triangulations of $\mathbf{S_{2p-dig}}$ and $\mathbf{S_{4p-sphere}}$ (up to isotopy)}
\label{2p+4p}
\end{table}
\end{proof}

We next discuss $W_\mathbf{S}$ for the surface $\mathbf{S}$ which is not one of the twice-punctured monogon, the once-punctured or twice-punctured digon, and the forth-punctured sphere. Triangulations of such surface can be obtained by gluing a number of pieces $\triangle_i$($i=1,\cdots,6$).

For a block shown in Figure \ref{piece0} obtained by gluing a piece $\triangle_3$ to a piece $\triangle_i(i=2,3)$, there is a 2-cycle in the associated quiver which needs to be deleted. We denote this block as the piece $\triangle_0$.
The unlabeled exterior side represents that it can be a matched side or a frozen side.
\begin{figure}[H]
\centering
\includegraphics[width=2.5cm]{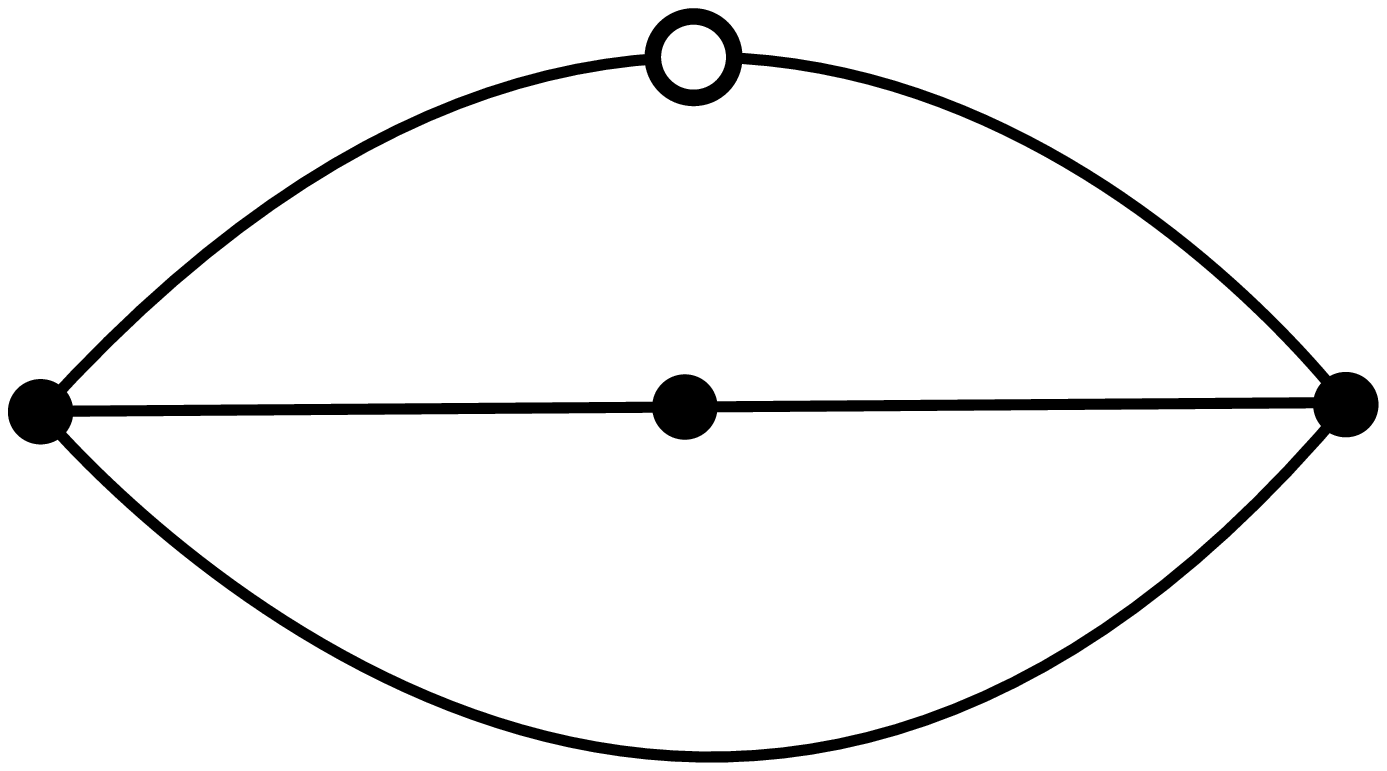}
\caption{Piece $\triangle_0$}
\label{piece0}
\end{figure}

By the relation between the triangulation and the associated exchange cluster quiver, we give the number of arrows corresponding to each piece shown in Table \ref{num-ori-tri}, where the negative integer means that when there is a piece $\triangle_0$, the number of arrows should be decreased by 2. Let $t_i$($i=0,1,\cdots,6$) be the number of pieces $\triangle_i$ in $T$.
\begin{table}[H]
\centering
\includegraphics[width=16cm]{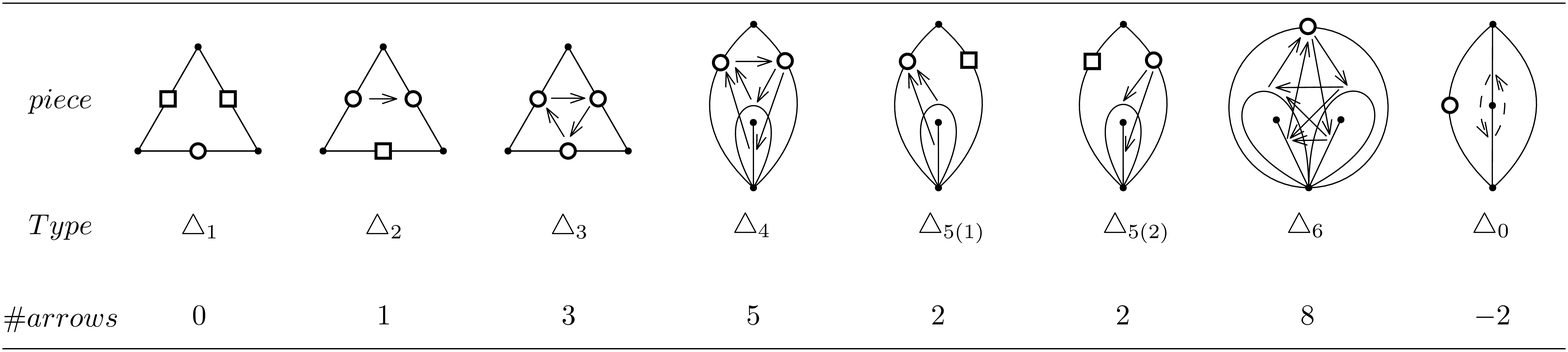}
\caption{The number of arrows for each piece with respect to $Q$}
\label{num-ori-tri}
\end{table}

\begin{Proposition}\cite{BY}
For a triangulated surface, which is none of a twice-punctured monogon, once-punctured digon or a forth-punctured sphere, let $t_i$ be the number of the piece $\triangle_i$ in a triangulation $T$ of $\mathbf{S}$. Then the calculation formula for the number of arrows of $Q$ associated with $T$ is $t=t_2+3t_3+5t_4+2t_5+8t_6-2t_0$.
\end{Proposition}
\begin{Remark}
  In \cite{BY}, the calculation formula for the number of arrows of $Q$ is represented as $e(Q_T)=3f(T)+w(T)-2d_{neg}(T)-s_f(T)-2s_w(T)$, where the definitions of symbols in the calculation formula can be found in \cite{BY}. Then we have that $e(Q_T)$ is the number of arrows $t$, $f(T)$ equals to $t_3+2t_4+t_5+3t_6$, $w(T)$ equals to $t_2+t_5$, $d_{neg}(T)$ equals to $t_0$, $s_f(T)$ equals to $t_4+t_6$, $s_w(T)$ equals to $t_5$.
\end{Remark}
\begin{Corollary}
For a triangulated surface $\mathbf{S}$ with $n$ arcs and $c$ boundary points, which is none of a twice-punctured monogon, once-punctured digon or a forth-punctured sphere, let $t_i$ be the number of the piece $\triangle_i$ in a triangulation $T$ of $\mathbf{S}$. Then \textbf{the calculation formula for the number of arrows of $\mathbf{Q}$} associated with $T$ is $t=2n-c+t_1-t_4-2(t_5+t_0)-t_6$.
\label{calculation}
\end{Corollary}
\begin{proof}
From Figure \ref{ori+tri}, according to the relation between the number of boundary points and the number of frozen sides in different kinds of pieces, it is easy to see that
$$c=2t_1+t_2+t_5.$$

Since except for the interior sides in pieces $\triangle_4$, $\triangle_5$ and $\triangle_6$, an arc in $T$ is the common side of two ideal triangles, which is obtained by gluing two matched sides together. Then according to the relation between the number of arcs and the number of matched edges in different kinds of pieces, we have the following formula, where the left side is the number of arcs obtained by gluing two matched sides, that is, the number of arcs minus the sum of the numbers of interior sides in pieces $\triangle_4$, $\triangle_5$ and $\triangle_6$, and the right side is the half number of matched sides before gluing.
$$n-2(t_4+t_5)-4t_6=(t_1+2t_2+3t_3+2t_4+t_5+t_6)/2.$$

Combining the above proposition, we have the calculation formula for the number of arrows of $Q$.
\end{proof}

\subsection{The maximum and minimum numbers of arrows of $Q$ arising from surfaces}
Let $t_{max}$ and $t_{min}$ be the maximum and minimum numbers of the distribution set $W_\mathbf{S}$ for a surface $\mathbf{S}$ respectively.

Corollary 4.4 of \cite{BY} shows that $t_{max}=2n-\frac{c+m}{2}$ holds if and only if the surface is not one of the once-punctured digon, triangle or rectangle and the twice-punctured monogon.
Then for a surface which is none of a twice-punctured monogon, a once-punctured or twice-punctured digon, and a forth-punctured sphere, we have
$$ t_{max}=
\begin{cases}
3 & \text{g=0, b=1, p=1, c=3}\\
5 & \text{g=0, b=1, p=1, c=4}\\
2n-\frac{c+m}{2} & \text{otherwise}
\end{cases},$$
where $m$ is the number of boundary components with odd boundary points, $c$ is the number of boundary points and $n$ is the number of arcs.
Next, we will give $t_{min}$ for a surface $\mathbf{S}$.

Some notations are necessary before creating a triangulation equipped with $t_{min}$ arrows.
For the surface with $g$ genus, the triangulation is represented by a plane graph obtained by \textbf{sniping} the surface from a marked point into a $4g-$polygon.

We call a \textbf{borderline} of a surface as a general term for boundaries of the surface and boundaries obtained by sniping the surface. For example, the triangulation of a torus is represented by a plane graph obtained by sniping it into a rectangle, then the borderline of a torus is the boundary of that rectangle. Besides for a surface with boundaries and zero genus, a borderline is a boundary component.

The graph shown in Figure \ref{many} as a part of some triangulation with $p$ punctures is called \textbf{a family of pieces $\mathbf{\triangle_0}$}. When there is only one puncture, the  family of pieces $\triangle_0$  is actually a piece $\triangle_0$.

\begin{Notation}
  Let $\alpha$ be an arc or a side on the boundary in a triangulated surface. We say that {\bf $\alpha$ is replaced  with a family of pieces $\triangle_0$} if $\alpha$ is glued with one of the two exterior sides of a family of pieces $\triangle_0$ and after this gluing, any contractible digon that could appear after gluing is deleted its one side.
\end{Notation}
\begin{figure}
\centering
\includegraphics[height=2.7cm]{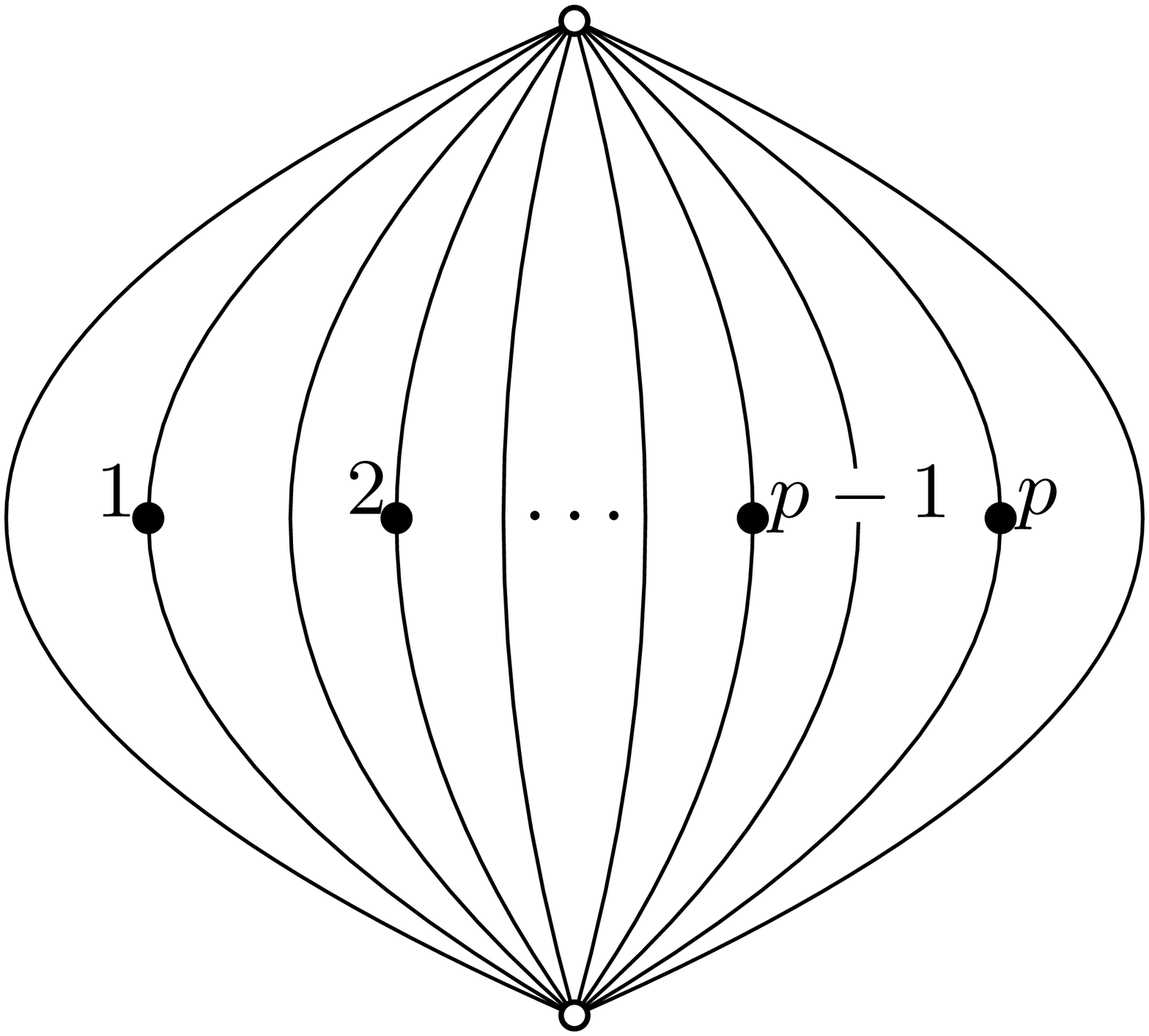}
\caption{A family of pieces $\triangle_0$}
\label{many}
\end{figure}
Let $\mathbf{S'}$ be a surface obtained by deleting all the punctures of the surface $\mathbf{S}$. Then by replacing an arc or a side on the boundary with a family of pieces $\triangle_0$, we will obtain a triangulated $\mathbf{S}$ from a triangulated $\mathbf{S'}$.
\begin{Notation}[Labeling procedure]
  If two borderlines have a common point, then this point is unique so label it by 0. Otherwise, randomly choose a point on each borderline and label it by 0. After that, label the other points $1,2,\cdots$ clockwise on each borderline.
\end{Notation}
\begin{Notation}[Linking procedure]
  After labeling procedure for a surface, if there exists an interior borderline, first choose a point of degree 2 on the exterior borderline and link it to all the points of degree 2 on the interior borderlines. Next link each point of degree 2 on the exterior borderline to suitable points on the interior borderline, make sure that arcs are compatible. Otherwise, link the point labeled by zero to points labeled by $2,3,\cdots,c-2$($c\geq4$).
\end{Notation}
\begin{Example}
  For a surface where $g=1,b=1,p=1,c=3$, proceed labeling procedure and linking procedure on it, we obtain a surface that has not been completely triangulated shown in below.
  \begin{figure}[H]
\centering
\includegraphics[height=1.8cm]{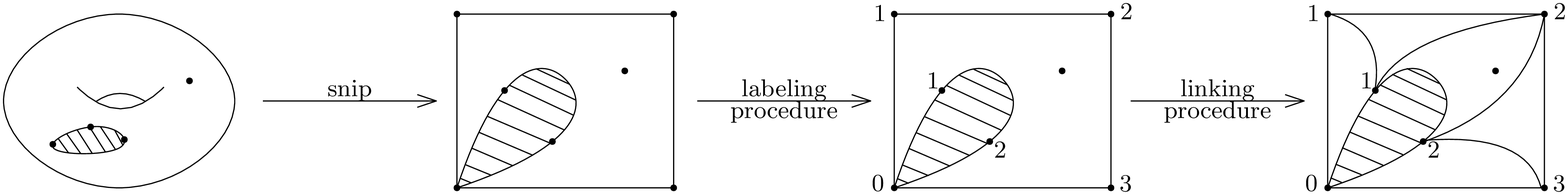}
\end{figure}
\end{Example}
By processing labeling procedure and linking procedure on a surface, it is clear that for the surface with more than one borderline, the degree of any boundary point is at least 3, so that the number of pieces $\triangle_1$ is zero.

According to the calculation formula of the number of arrows of $Q$, the coefficient of $t_i$($i=0,4,5,6$) is negative, while that of $t_1$ is positive.
To create a triangulation equipped with minimum number of arrows, we wish $t_i(i=0,4,5,6$) as small as possible and $t_1$ as large as possible.
Besides, because the coefficients of $t_5$ and $t_0$ are less then that of $t_4$ and $t_6$, and flipping from a piece $\triangle_5$ to a piece $\triangle_0$ will not change the number of arrows, so we might as well let $t_0$ the most.
According to the number of punctures lied in the interior of $\triangle_0$, we have $p\geq t_0+t_4+t_5+2t_6$, so that ${(t_0)}_{max}=p$.
Thus the optimal condition for $t_{min}$ is that $t_1=0$ and $t_0=p$.

Note that except for the boundary point which is the junction point of two frozen sides in a piece $\triangle_1$, the others are of degree 3. Then the equivalent condition of that $t_1=0$ is that the degree of any boundary point is at least 3, denoted by \emph{Min1}.

It is easy to see that $t_0=p$ $\Leftrightarrow$ each puncture lies in the interior of a piece $\triangle_0$ $\Leftrightarrow$ points on the side of any piece $\triangle_0$ are boundary points.

\begin{Lemma}
  For a surface $\mathbf{S}$ with $p$ punctures, there is a triangulation with $t_0=p$ if and only if there are two boundary points labeled the same on the different borderline or labeled different after proceeding labeling procedure on $\mathbf{S}$.
\end{Lemma}
\begin{proof}
  ($\Leftarrow$) Because there are two boundary points $k$ and $k'$ labeled the same on the different borderline or labeled different, then there exists a triangulation of $\mathbf{S'}$ with an arc whose endpoints are $k$ and $k'$. Thus we can obtain a triangulation of
  $\mathbf{S}$ with $t_0=p$ by replacing that arc with a family of pieces $\triangle_0$.

  ($\Rightarrow$) If such boundary points do not exist, $\mathbf{S}$ can only be a monogon with $p$ punctures. Then the triangulation with $t_0=p$ can only be obtained by replacing the side of $\mathbf{S'}$ with a family of pieces $\triangle_0$. at this time, there is a contractible loop in that triangulation, which is a contradiction.
\end{proof}
The equivalent condition in the lemma above is denoted by \emph{Min2}.

Thus the optimal situation for the triangulation equipped with $t_{min}$ arrows is that \emph{Min1} and \emph{Min2} hold simultaneously.
\begin{Lemma}
Assume that $\mathbf{S}$ is a surface which is none of a twice-punctured monogon, a once-punctured or twice-punctured digon, and a forth-punctured sphere. Let $t_{min}$ be minimum number of the distribution set $W_\mathbf{S}$, then
$$t_{min}=
\begin{cases}
2n-2p+4 & \text{g=0, b=0, p$\geq$5, c=0}\\
2n-2p & \text{g=0, b=1, p$\geq$3, c=1}\\
2n-c-2p+max\{0,2-p\} & \text{g=0, b=1, p$\geq$0, c$\geq$2}\\
2n-2p+2 & \text{g$\geq$1, b=0, p$\geq$1, c=0}\\
2n-c-2p & \text{otherwise.}
\end{cases},
\label{num}
$$
where $c$ is the number of boundary points, $p$ is the number of punctures and $n$ is the number of arcs.
\label{max+min}
\end{Lemma}
\begin{proof}
Our main idea is to create a special triangulation equipped with $t_{min}$ arrows. Then we will discuss case by case for different type of surfaces.

\vspace{1 ex}
\textbf{Case1}: $\mathbf{S_1}$ where $g=0, b=0, p\geq5, c=0$.

Since there is no boundary, so that \emph{Min1} holds but \emph{Min2} does not.
Then a family of pieces $\triangle_0$ can only replace an arc whose endpoints are punctures, thus the maximum $t_0$ is $p-2$. By Minimum Construction $\uppercase\expandafter{\romannumeral1}$, we can obtain a triangulation of $\mathbf{S_1}$ equipped with $t_{min}$ arrows which is $2n-2p+4$.

\emph{Minimum Construction \uppercase\expandafter{\romannumeral1}}
\par $\bullet$ \emph{Choose two punctures and connect them by an arc.}
\par $\bullet$ \emph{Construct a family of pieces $\triangle_0$ by the remaining unconnected $p-2$ punctures.}
\par $\bullet$ \emph{Replace the arc with a family of pieces $\triangle_0$.}
\vspace{1 ex}

\textbf{Case2}: $\mathbf{S_2}$ where $g=0, b\geq1, p\geq0, c\geq1$.

we first discuss when conditions \emph{Min1} and \emph{Min2} can hold simultaneously for $\mathbf{S_2}$ after proceeding labeling procedure and linking procedure on $\mathbf{S'_2}$ (recall that $\mathbf{S'}$ is the surface obtained by deleting all the punctures of $\mathbf{S}$).

When $b\geq2$, obviously \emph{Min1} and \emph{Min2} can hold simultaneously, for the reason that after proceeding labeling procedure, \emph{Min2} holds; after proceeding linking procedure, \emph{Min1} holds.

When $b=1, c\geq2$, obviously \emph{Min2} holds.
If $c=3$ or $c=2$, there are at least two boundary points of degree 2 of $\mathbf{S'_2}$. It is easy to see that under the condition \emph{Min2}, a method to achieve condition \emph{Min1} is to replace two sides on the boundary with two pieces $\triangle_0$.
If $c\geq 4$, there are at least two pieces $\triangle_1$ in $\mathbf{S'_2}$ after proceeding linking procedure. In order to achieve condition \emph{Min1}, pieces $\triangle_1$ need to be removed. Under the condition \emph{Min2}, a method to remove a piece $\triangle_1$ is to replace one frozen side of $\triangle_1$ with a piece $\triangle_0$.
Thus \emph{Min1} and \emph{Min2} hold simultaneously only if there are at least 2 punctures.

When $b=1, c=1$, at this time $p\geq 3$. it is easy to see that \emph{Min1} holds but \emph{Min2} does not.

Thus \emph{Min1} and \emph{Min2} cannot hold simultaneously for a once-punctured polygon where $c\geq2$ or a monogon with $p$ ($p\geq3$) punctures.

For the surface where $g=0,b=1,p\geq3,c=1$, \emph{Min2} does not hold so that $(t_0)_{max}=p-1$.
We create the triangulation equipped with $t_{min}$ arrows shown in Figure.\ref{pfmax} where $t_0=p-1, t_4=1$, thus $t_{min}$ is $2n-2p$.
\begin{figure}[H]
\centering
\includegraphics[width=4cm]{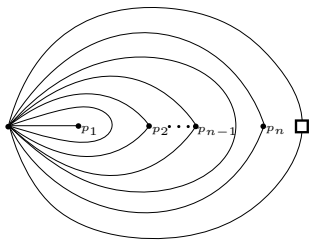}
\caption{A triangulation equipped with $t_{min}$ arrows of $\mathbf{S}$ where $g=0,b=1,p\geq3,c=1$}
\label{pfmax}
\end{figure}

Note that for the surface where $g=0,b=1,p=1,c\geq3$ (recall that $c=2$ has been discussed and we do not allow a monogon with one puncture), the triangulation with $t_1=1$ and $t_0=1$ is equipped with less arrows than that with $t_1=0$ and $t_0=0$, and the other triangulations are obvious equipped with more arrows. Thus we will create a triangulation with $t_1=1$ and $t_0=1$ for such surface.

Then for the surface $\mathbf{S_2}$ which is not one where $g=0,b=1,p\geq3,c=1$, by Minimum Construction $\uppercase\expandafter{\romannumeral2}$, we obtain a triangulation equipped with $t_{min}$ arrows which is $2n-c-2p+max\{0,2-p\}$ when $g=0, b=1, p\geq0, c\geq2$ or $2n-c-2p$ when $g=0, b\geq2, p\geq0, c\geq2$.

\emph{Minimum Construction \uppercase\expandafter{\romannumeral2}}
\par $\bullet$ \emph{Proceed labeling procedure and linking procedure on $\mathbf{S'}$.}
\par $\bullet$ \emph{When there is only one borderline, if $p=1$, replace a frozen side whose endpoints are labeled by $0$ and $1$ with a piece $\triangle_0$. If $p\geq2$, replace two frozen sides whose endpoints are labeled by $0$ and $1$, $0$ and $c-1$ respectively with two pieces $\triangle_0$}
\par $\bullet$ \emph{If there are remaining unconnected punctures, construct a family of pieces $\triangle_0$ by these punctures and then replace an arc whose endpoints are not punctures with it.}
\par $\bullet$ \emph{Supplement arcs to get a triangulation.}
\vspace{1 ex}

\textbf{Case3}: $\mathbf{S_3}$ where $g\geq1, b=0, p\geq1, c=0$.

The triangulation of a surface where $g\neq0$ is generally represented in a plane graph by snipping the surface. Since there is no boundary, so that conditions \emph{Min1} holds but \emph{Min2} does not.
Because points on the borderline of the snipped surface are the same puncture, there are at most $p-1$ pieces $\triangle_0$. By Minimum Construction $\uppercase\expandafter{\romannumeral3}$, we obtain a triangulation equipped with $t_{min}$ arrows where $t_{min}=2n-2p+2$.

\emph{Minimum Construction \uppercase\expandafter{\romannumeral3}}
\par $\bullet$ \emph{Snip the surface from a puncture into a plane graph whose underlying graph is a $4g$-sided polygon and then triangulate it.}
\par $\bullet$ \emph{Construct a family of pieces $\triangle_0$ by the remaining $p-1$ punctures and then replace an arc whose endpoints are labeled different with it.}
\par $\bullet$ \emph{Supplement arcs to get a triangulation.}
\vspace{1 ex}

\textbf{Case4}: $\mathbf{S_4}$ where $g\geq1,\ b\geq1,\ p\geq0,\ c\geq1$

The triangulation of $\mathbf{S_4}$ is represented in a plane graph snipping from a boundary point. In this situation, there are at least 2 borderlines. By the discussion in Case 2, it is known that \emph{Min1} and \emph{Min2} hold simultaneously after proceeding labeling procedure and the linking procedure. Thus by Minimum Construction $\uppercase\expandafter{\romannumeral4}$, we obtain a triangulation equipped with $t_{min}$ arrows where $t_{min}=2n-c-2p$.

\emph{Minimum Construction \uppercase\expandafter{\romannumeral4}}
\par $\bullet$ \emph{Snip the surface from a boundary point into a plane graph whose underlying graph is a 4g-sided polygon with $b$ holes.}
\par $\bullet$ \emph{Proceed the Minimum Construction \uppercase\expandafter{\romannumeral2}.}
\end{proof}

\begin{Proposition}
  The maximum number $t_{max}$ and minimum number $t_{min}$ of the distribution set $W_\mathbf{S}$ for a surface $\mathbf{S}$ are given respectively in the following table, where $m$ is the number of boundary components with odd boundary points, $c$ is the number of boundary points, $p$ is the number of punctures and $n$ is the number of arcs.
\begin{table}[H]
\centering
\includegraphics[width=16cm]{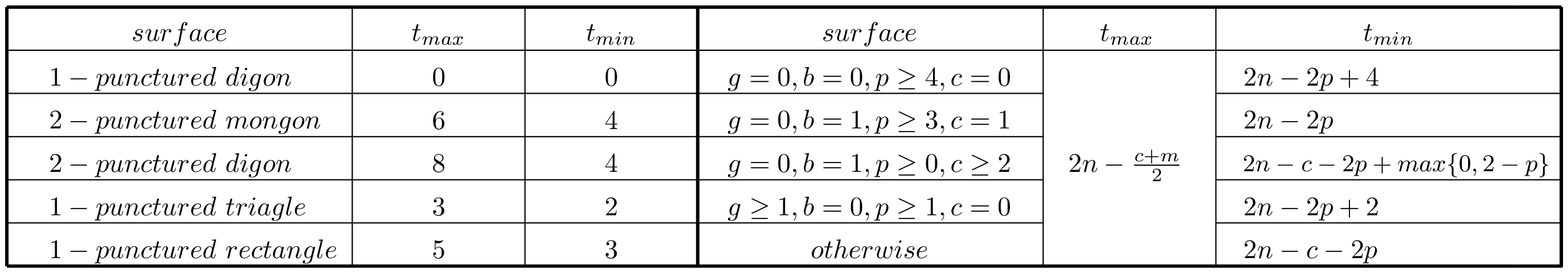}
\caption{The $t_{max}$ and $t_{min}$ of $W_\mathbf{S}$}
\label{type}
\end{table}
\vspace{-0.7cm}
\label{num-t}
\end{Proposition}

\begin{proof}
   The maximum number $t_{max}$ is given in Corollary 4.4 of \cite{BY}.
   Combining Lemma \ref{78}, Lemma \ref{bulianxu}, Lemma \ref{max+min}, the proof follows easily.
\end{proof}

\subsection{The existence of a complete walk}

The continuity proof of Theorem \ref{main} with respect to exchange cluster quivers arising from surfaces is based on the following Lemma \ref{perfect1}, the main ideal of continuity proof is to establish the existence of a complete walk for a surface.
We first introduce some notations, which are used throughout the proof of Lemma \ref{perfect1}.

\begin{Definition}
We call a walk a \textbf{complete walk} (or an \textbf{extended complete walk}) in the graph $\mathbf{E^\circ(S,M)}$ of $\mathbf{S}$ if this walk is acyclic and the set of the numbers of arrows for exchange cluster quivers (or (extended exchange cluster quivers) associated with triangulations on this walk  equals to exactly the distribution set $W_\mathbf{S}$ (or $\tilde{W}_\mathbf{S}$).
\label{perfect}
\end{Definition}
On a complete walk, triangulations equipped with the same number of arrows may appear.

Let $l(T)$ be a walk starting from $T$, then a composite walk $l_m\circ\cdots\circ l_1(T_1)$ is
\begin{figure}[H]
  \centering
  \includegraphics[width=7.5cm]{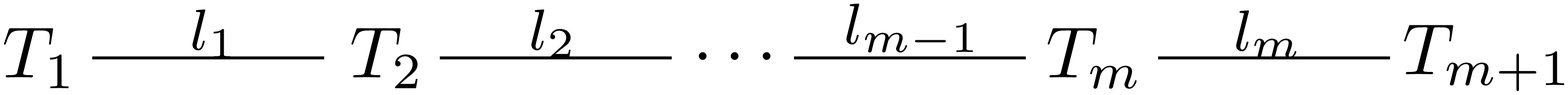}
\end{figure}
\noindent
where the triangulation $T_i$ is called the start-point of walk $l_i$ and the end-point of walk $l_{i-1}(i=2,\cdots,m)$, $T_1$ is called the start-point of the composite walk, and $T_{m+1}$ is called the end-point of the composite walk.
Let $\Phi_i$ be the set $\{t\ |\ t=|Q_1|$, where $Q$ is the cluster quiver associated with the triangulation on the walk $l_i(T_i)$ without the start-point$\}$ where $i=2,3,\cdots,m$. And $\Phi_1=\{t\ |\ t=|Q_1|$, where $Q$ is the cluster quiver associated with the triangulation on the walk $l_1(T_1)\}$.

Some types of flips listed in the Table \ref{walk-f} will appear on a complete walk, where $f_{K,(k,k')}$ represent a flip of type $K$ in the graph in which points labeled by $k$ and $k'$ are located, and unlabeled side represents that it can be a frozen or a mutable side.
In each row, we list types of flips, two local graphs of triangulations related by that flip, and the absolute value of the change in the number of arrows of exchange cluster quivers under that flip.
We use $\sum_{i=a}^{i=b}f_{K,(i,s)}$ to indicate $f_{K,(b,s)} \circ\cdots\circ f_{K,(a,s)}$, and $f_{K,(k,k')}|_{cond}$ to indicate that do the flip only if the condition is satisfied.
\vspace{-0.5cm}
\begin{table}[H]
\centering
\includegraphics[width=16cm]{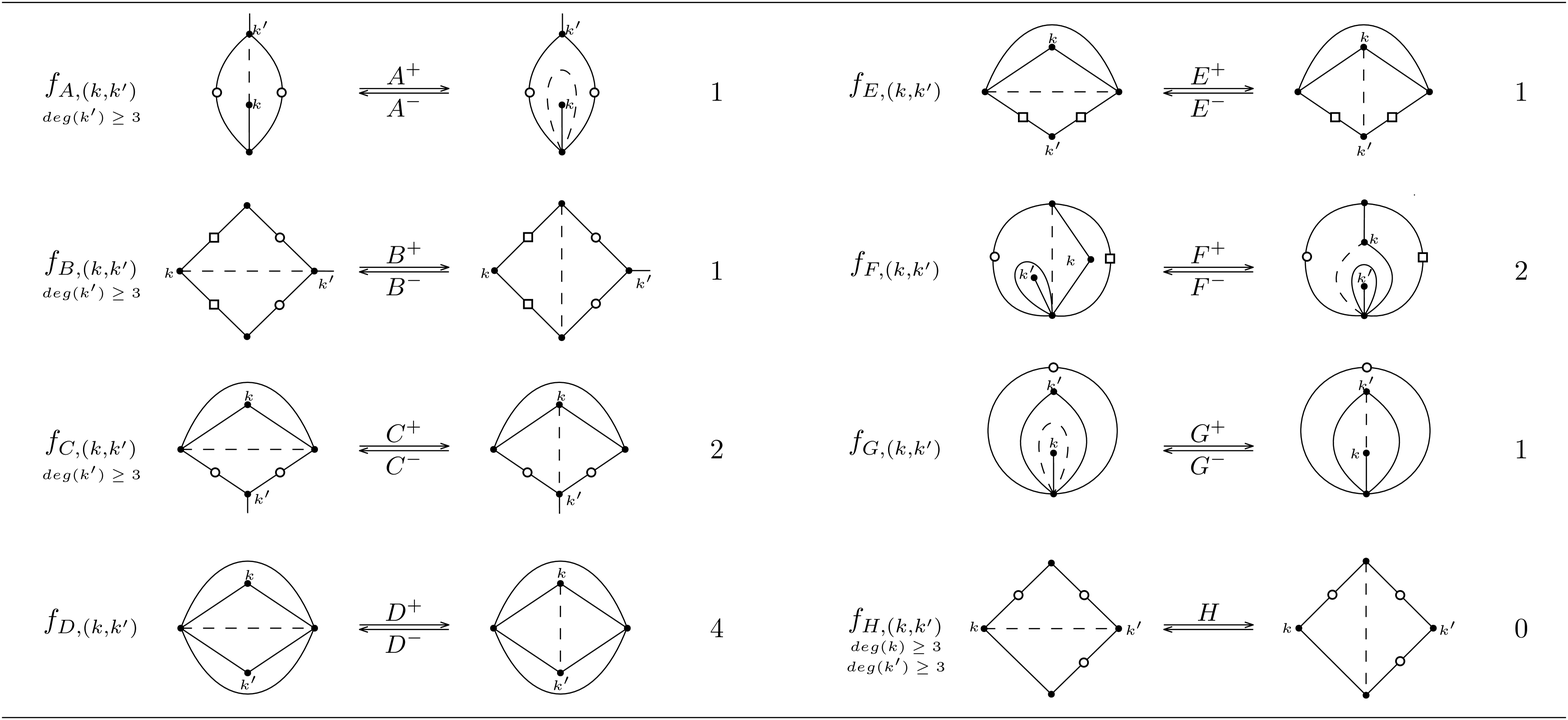}
\caption{Some types of flips}
\label{walk-f}
\end{table}
\vspace{-0.5cm}

\begin{Lemma}
There is a complete walk in graph $\mathbf{E^\circ(S,M)}$ of a surface which is not a twice-punctured digon or a forth-punctured sphere.
\label{perfect1}
\end{Lemma}

\begin{proof}
We will proceed case by case for different types of surfaces.

For the twice-punctured monogon, a complete walk is $T_1-T_2-T_3$, where $T_i$ is the triangulation shown in Figure \ref{2p-mon}. For the once-punctured digon, a complete walk is a vertex in $\mathbf{E^\circ(S,M)}$.

\textbf{Case1:} For $\mathbf{S_1}$ where $g=0, b=0, p\geq5, c=0$, let $T_1$ be a triangulation shown below.
Then two triangulations $T_2=l_1(T_1)$ and $T_3=l_2(T_2)$ shown below will appear on the following walks in (\ref{1.1}), (\ref{1.2}), (\ref{1.3}):
\begin{figure}[H]
\centering
\includegraphics[width=10cm]{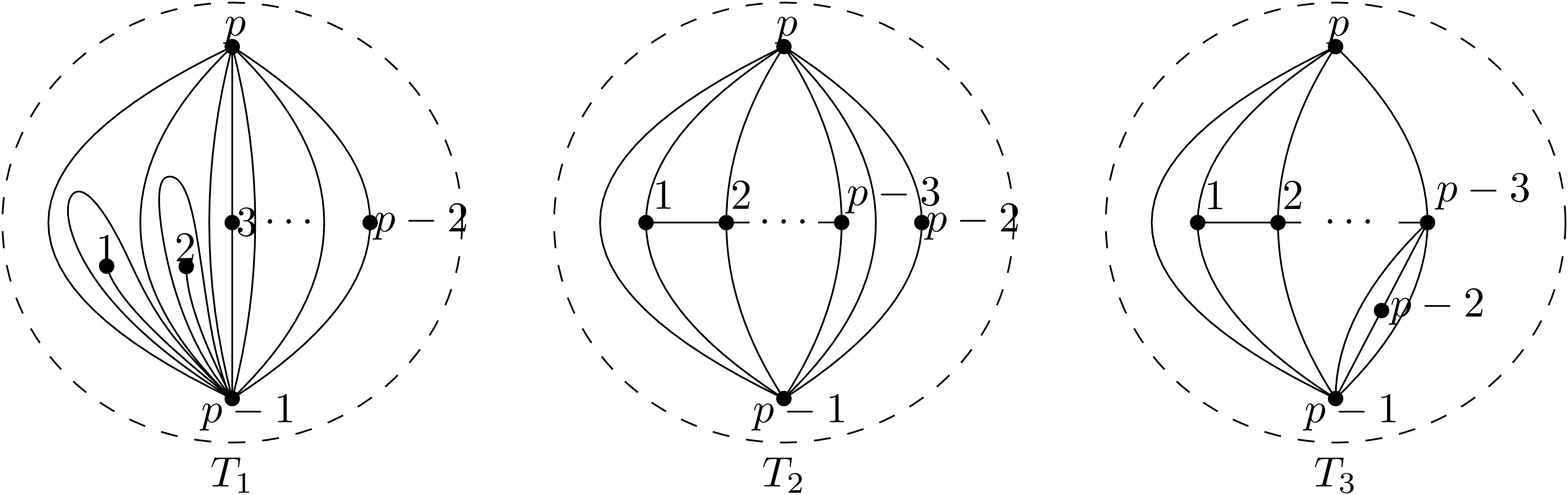}
\end{figure}
\vspace{-0.7cm}
Firstly, we have
\begin{equation}\label{1.1}
l_1(T_1)=(\sum\nolimits_{i=3}^{i=p-3}f_{C^+,(i,i-1)})|_{p\geq6}\circ f_{D^+,(1,2)}\circ f_{A^-,(2,p)}\circ f_{A^-,(1,p)}(T_1),
\end{equation}
then $\Phi_1=[4p-8,4p-6]\cap\mathbb{N}\cup\{4p-4+2s|s=0,\cdots,p-5\}$;

Secondly, we have
\begin{equation}\label{1.2}
l_2(T_2)=f_{C^-,(p-2,p)}\circ f_{C^+,(p-2,p-3)}(T_2),
\end{equation}
then $\Phi_2=\{6p-14,6p-12\}$;

Thirdly, we have
\begin{equation}\label{1.3}
l_3(T_3)=\sum\nolimits_{i=1}^{i=p-4}f_{C^-,(i,i+1)}\circ f_{A^+,(p-2,p-3)}(T_3),
\end{equation}
then $\Phi_3=\{6p-13-2s|s=0,\cdots,p-4\}$.

Thus we obtain $\Phi_1\cup\Phi_2\cup\Phi_3=[4p-8,6p-12]$.

It is easy to see that on this composite walk, there are only two triangulations $T_2$ and $T_3$ equipped with the same number of arrows , which is $6p-14$. While they are obviously not isotopy.

Hence $l_3\circ l_2\circ l_1(T_1)$ is a complete walk for $\mathbf{S_1}$.

\textbf{Case2:} For $\mathbf{S_2}$ where $g=0, b\geq1, p\geq3, c=1$, let $T_1$ be a triangulation equipped with $t_{min}$ arrows shown below. Then two triangulations $T_2=l_1(T_1)$ and $T_3=l_2(T_2)$ shown below will appear on the following walks in (\ref{2.1}), (\ref{2.2}), (\ref{2.3}):
\begin{figure}[H]
\centering
\includegraphics[width=10cm]{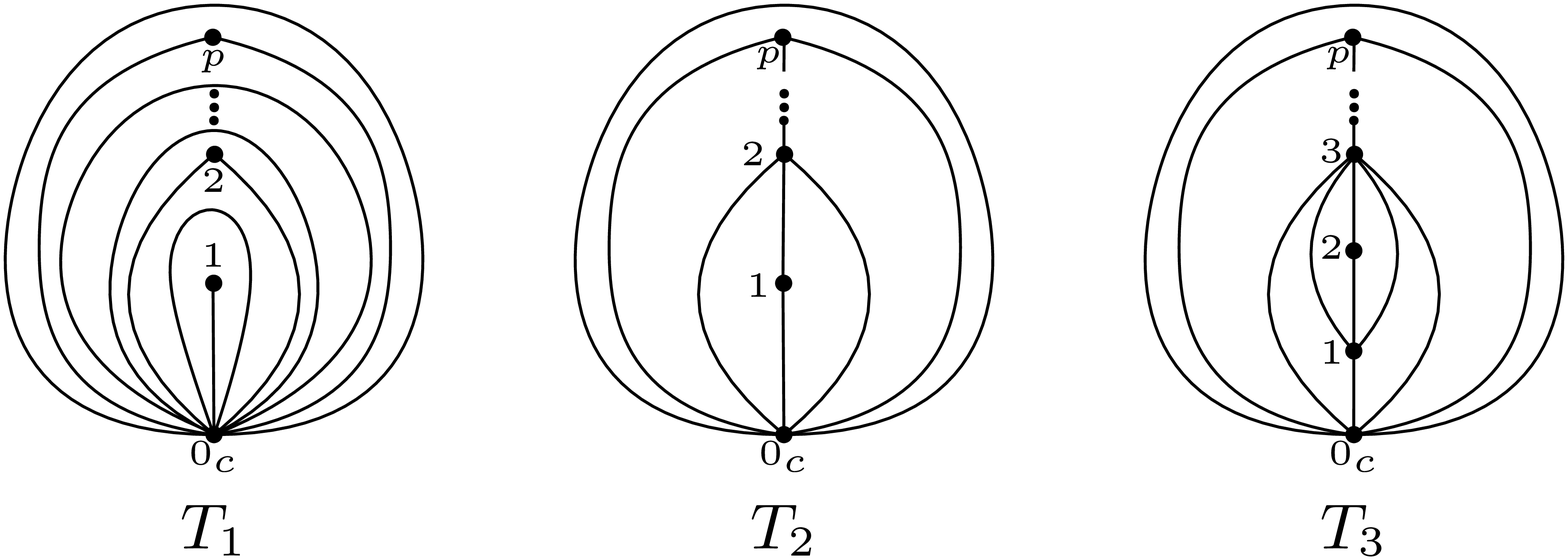}
\end{figure}
\vspace{-0.7cm}

Firstly, we have
\begin{equation}\label{2.1}
l_1(T_1)=\sum\nolimits_{i=3}^{i=p}f_{C^+,(i,i-1)}\circ f_{G^+,(1,2)}(T_1),
\end{equation}
then $\Phi_1=\{4p-4,4p-3+2s|s=0,\cdots,p-2\}$;

Secondly, we have
\begin{equation}\label{2.2}
l_2(T_2)=f_{C^-,(2,0_c)}\circ f_{C^+,(1,3)}(T_2),
\end{equation}
then $\Phi_2=\{6p-7,6p-5\}$;

Thirdly, we have
\begin{equation}\label{2.3}
l_3(T_3)=(\sum\nolimits_{i=p}^{i=4}f_{C^-,(i,i-1)})|_{p\geq4}\circ f_{C^-,(1,0_c)}\circ f_{A^+,(2,1)}(T_3),
\end{equation}
then $\Phi_3=\{6p-6-2s|s=0,\cdots,p-2\}$.

Thus we obtain $\Phi_1\cup\Phi_2\cup\Phi_3=[4p-4,6p-5]$.

It is easy to see that on this composite walk, there are only two triangulations $T_2$ and $T_3$ with the same number of arrows, which is $6p-7$. While they are obviously not isotopy.

Hence $l_3\circ l_2\circ l_1(T_1)$ is a complete walk for $\mathbf{S_2}$ .

\textbf{Case3:} \textbf{1)} For the surface where $g=0,b=1,p=1,c=3$, let $T_l$ be a triangulation shown on the left side of the figure below, then $f_{E^+,(1,1_c)}(T_l)$ is a complete walk.

\textbf{2)} For the surface where $g=0,b=1,p=1,c=4$, let $T_l$ be a triangulation shown on the right side of the figure below, then $f_{E^-,(1,3_c)}\circ f_{C^+,(1,2_c)}(T_r)$ is a complete walk.
\begin{figure}[H]
\centering
\includegraphics[width=6.5cm]{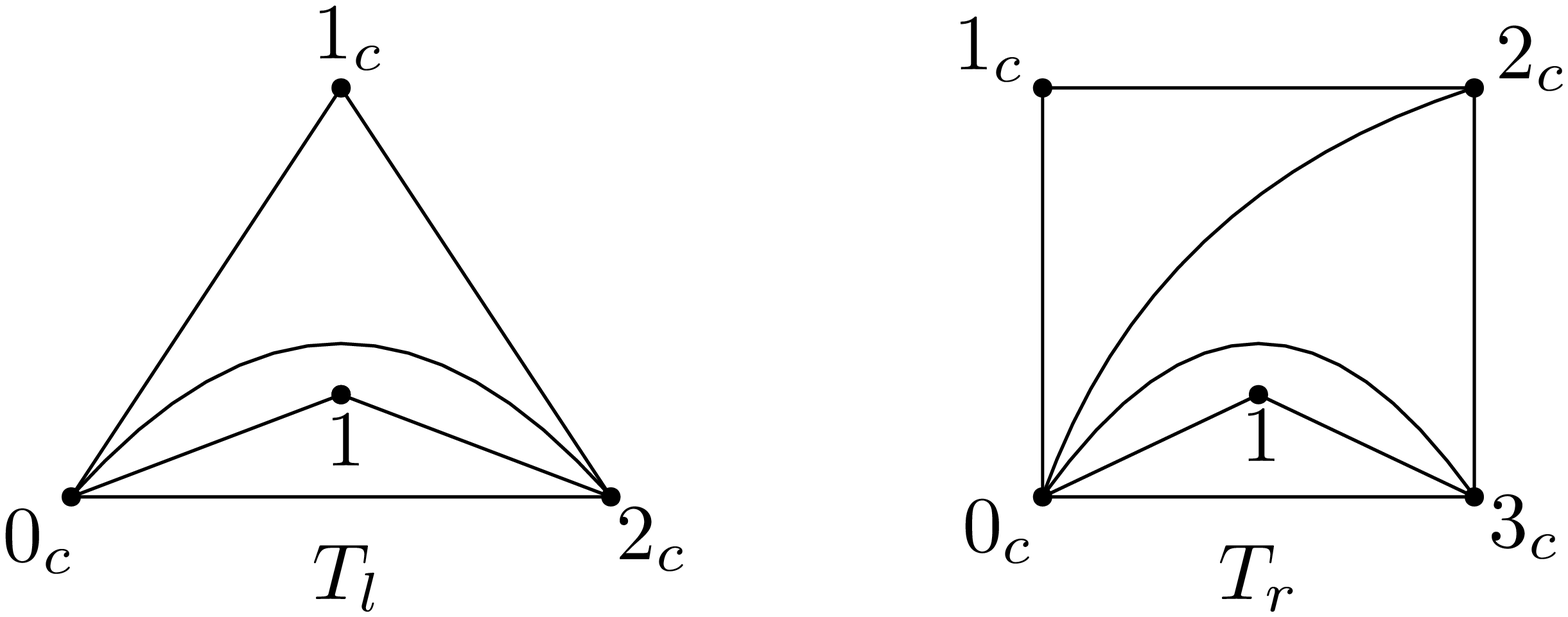}
\end{figure}

\textbf{3)} For the surface where $g=0,b=1,p=0,c\geq2$, $t_{min}=c-4$ and $t_{max}=\frac{3}{2}c-6-\frac{m}{2}$. It is easy to see that the triangulation $T$ obtained by processing linking procedure for the surface can translate into the triangulation equipped with $t_{max}$ arrows by $\frac{c-m}{2}-2$ times flips of type $B^+$. Thus $l_{p=0}(T)=(\sum\nolimits_{i=1}^{i={\frac{c-m}{2}}-2}f_{B^+,((2i+1)_c,0_c)})|_{c\geq6}(T)$ is a complete walk.

\textbf{4)} For the surface where $g=0,b=1,p=1,c\geq5$, $t_{min}=c-1$ and $t_{max}=\frac{3}{2}c-\frac{m}{2}$. At this time, $c\geq5$ (recall that $c=3$ and $c=4$ have been discussed before).
Let $T_1$ be a triangulation equipped with $t_{min}$ arrows shown below.
Then a walk staring from $T_1$ is $l_1(T_1)=l_{p=0}\circ f_{E^-,(1,1_c)}\circ f_{C^+,(1,2_c)}(T_1)$,
so that $\Phi_1$ is $[c-1,\frac{3}{2}c-\frac{m}{2}-2]$.

We consider the local triangulation of $T_2=l_1(T_1)$ shown below containing the piece $\triangle_0$ in which the puncture labeled by 1 lies. In this rectangle, the degree of point labeled by $k$ is at least 3 (if the degree is 2, when $k$ is a boundary point, we have $c=4$, which is a contradiction; when $k$ is a puncture, while any puncture of degree 2 is in a piece $\triangle_0$, which is a contradiction), so that we can do the flip of type $C$.
Then the walk $l_2(T_2)=f_{A^+,(1,k)}\circ f_{C^-,(1,2_c)}\circ f_{C^+,(1,k)}(T_2)$ is to flip in this rectangle and keep the other part unchanged, so that $\Phi_2$ is $[\frac{3}{2}c-\frac{m}{2}-1,\frac{3}{2}c-\frac{m}{2}]$.
\begin{figure}[H]
\centering
\includegraphics[width=8cm]{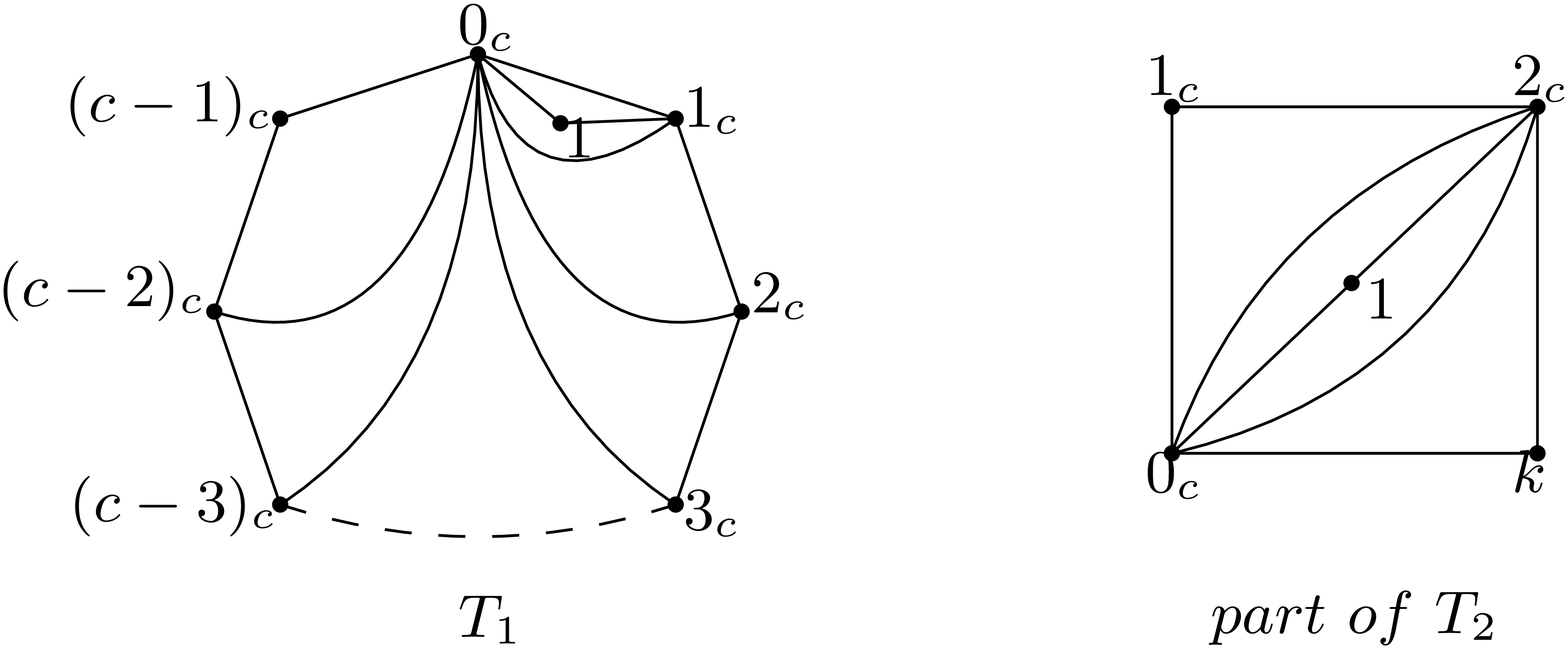}
\end{figure}
Thus we obtain $\Phi_1\cup\Phi_2=[c-1,\frac{3}{2}c-\frac{m}{2}]$.

On this composite walk, there are only two triangulations $T_2$ and $ f_{C^-,(1,2_c)}\circ f_{C^+,(1,k)}(T_2)$ equipped with the same number of arrows, which is $\frac{3}{2}c-\frac{m}{2}-2$. While they are obviously not isotopy.

Hence $l_2\circ l_1(T_1)$ is a complete walk.

\textbf{5)} For the surface where $g=0,b=1,p\geq3,c=2$, $t_{min}=4p-4$ and $t_{max}=6p-3$. Let $T_1$ be a triangulation shown below.
Then two triangulations $T_2=l_1(T_1)$ and $T_3=l_2(T_2)$ shown below will appear on the following walks in (\ref{3.4.1}), (\ref{3.4.2}), (\ref{3.4.3}):

\begin{figure}[H]
\centering
\includegraphics[width=12cm]{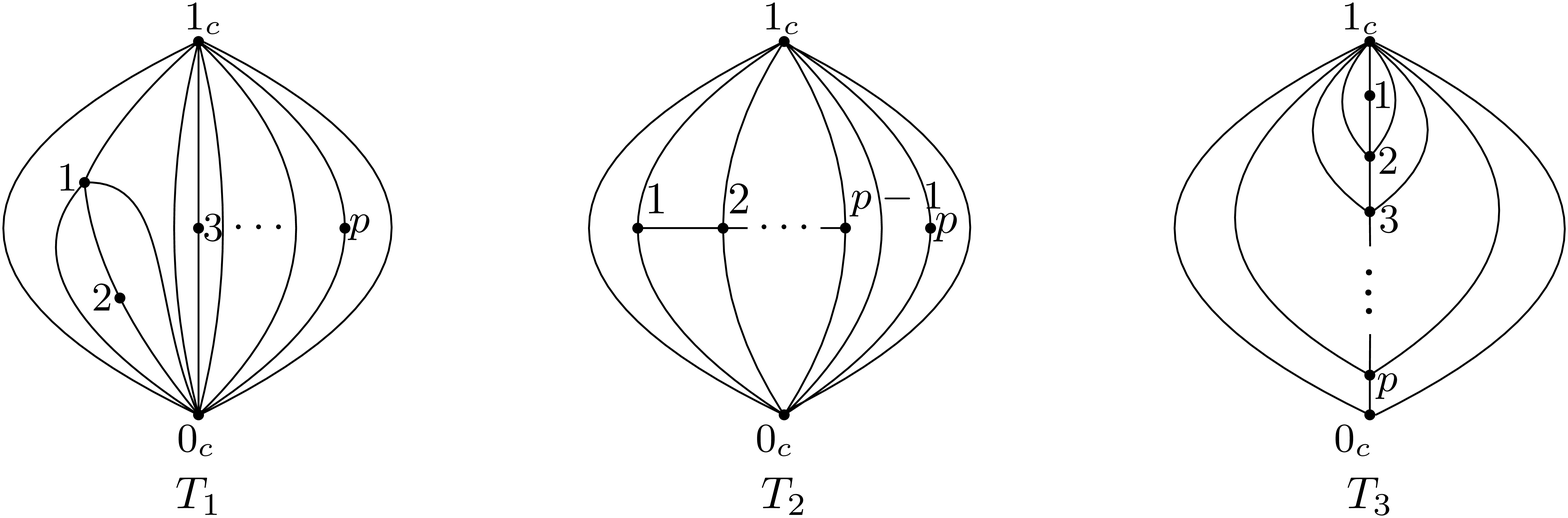}
\end{figure}
\vspace{-0.7cm}

Firstly, we have
\begin{equation}\label{3.4.1}
l_1(T_1)=(\sum\nolimits_{i=3}^{i=p-1}f_{C^+,(i,i-1)})|_{p\geq4}\circ f_{D^+,(1,2)}\circ f_{A^-,(2,1_c)}\circ f_{F^-,(1,2)}\circ f_{A^+,(2,1)}(T_1),
\end{equation}
then $\Phi_1=[4p-4,4p]\cup\{4p+2s|s=1,\cdots,p-3\}$;

Secondly, we have
\begin{equation}\label{3.4.2}
l_2(T_2)= \sum\nolimits_{i=2}^{i=p-1}f_{H,(i,0_c)}\circ f_{C^-,(1,0_c)}\circ f_{C^+,(p,p-1)}(T_2),
\end{equation}
then $\Phi_2=\{6p-6,6p-4\}$;

Thirdly, we have
\begin{equation}\label{3.4.3}
l_3(T_3)=f_{C^-,(2,1_c)} \circ (\sum\nolimits_{i=p}^{i=4}f_{C^-,(i,i-1)})|_{p\geq4}\circ f_{C^+,(1,3)}\circ f_{B^+,(0_c,p)}(T_3),
\end{equation}
then $\Phi_3=\{6p-3-2s|s=0,\cdots,p-2\}$.

Thus we obtain $\Phi_1\cup\Phi_2\cup\Phi_3=[4p-4,6p-3]$.

On this composite walk, there are $p$ triangulations equipped with the same number $6p-6$ of arrows, which are respectively $T_2,\ f_{C^-,(1,0_c)}\circ f_{C^+,(p,p-1)}(T_2),\ f_{H,(2,0_c)}\circ f_{C^-,(1,0_c)}\circ f_{C^+,(p,p-1)}(T_2),\cdots, T_3$. While they are obviously not isotopy.
Besides, two triangulations $f_{B^+,(0_c,p)}(T_3)$ and $f_{C^-,(p,p-1)}\circ f_{C^+,(1,3)}\circ f_{B^+,(0_c,p)}(T_3)$ equipped with the same number $6p-5$ of arrows are obviously not isotopy.

Hence $l_3\circ l_2\circ l_1(T_1)$ is a complete walk.

\textbf{6)} For the surface where $g=0,b=1,p\geq2,c=3$, $t_{min}=4p-3$ and $t_{max}=6p-2$. Let $T_1$ be a triangulation equipped with $t_{min}$ arrows shown below.
Then two triangulations $T_2=l_1(T_1)$ and $T_3=l_2(T_2)$ shown below will appear on the following walks in (\ref{3.5.1}), (\ref{3.5.2}), (\ref{3.5.3}):
\begin{figure}[H]
\centering
\includegraphics[width=10cm]{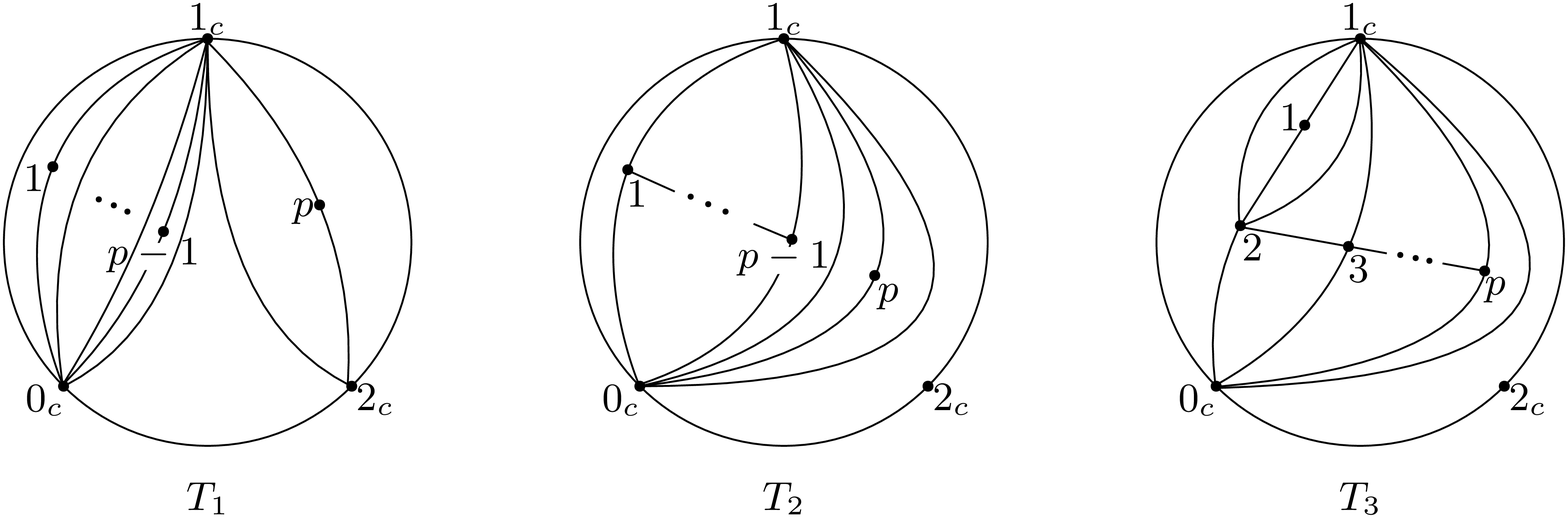}
\end{figure}
Firstly, we have
\begin{equation}\label{3.5.1}
l_1(T_1)=(\sum\nolimits_{i=3}^{i=p-1}f_{C^+,(i,i-1)})|_{p\geq4}\circ f_{D^+,(1,2)}\circ f_{E^-,(p,2_c)}\circ f_{C^+,(p,0_c)}(T_1),
\end{equation}
then $\Phi_1=[4p-3,4p-1]\cup\{4p+2+2s|s=0,\cdots,p-3\}$;

Secondly, we have
\begin{equation}\label{3.5.2}
l_2(T_2)=f_{C^-,(1,0_c)}\circ (f_{C^+,(p,p-1)})|_{p\geq3}(T_2),
\end{equation}
then $\Phi_2=\{6p-4,6p-2\}$;

Thirdly, we have
\begin{equation}\label{3.5.3}
l_3(T_3)=(f_{A^-,(1,1_c)} \circ\sum\nolimits_{i=p}^{i=3}f_{C^-,(i,i-1)})|_{p\geq3}\circ f_{A^+,(1,1_c)}(T_3),
\end{equation}
then $\Phi_3=\{4p,6p-3-2s|s=0,\cdots,p-2\}$.

Thus we obtain $\Phi_1\cup\Phi_2\cup\Phi_3=[4p-3,6p-2]$.

On this composite walk, there are only two triangulations $T_2$ and $T_3$ equipped with the same number of arrows, which is $6p-4$. While they are obviously not isotopy.

Hence $l_3\circ l_2\circ l_1(T_1)$ is a complete walk.

\textbf{7)} For the surface where $g=0,b=1,p\geq2,c=4$, $t_{min}=4p-2$ and $t_{max}=6p$. Let $T_1$ be a triangulation equipped with $t_{min}$ arrows shown below.
Then two triangulations $T_2=l_1(T_1)$ and $T_3=l_2(T_2)$ shown below will appear on the following walks in (\ref{3.6.1}), (\ref{3.6.2}), (\ref{3.6.3}):
\begin{figure}[H]
\centering
\includegraphics[width=13cm]{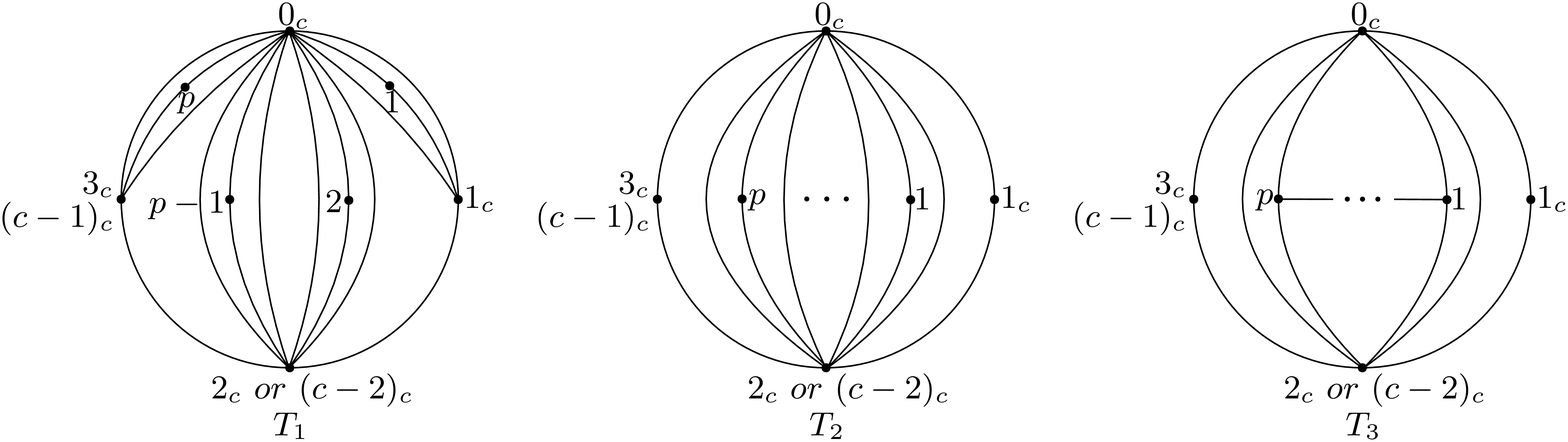}
\end{figure}

Firstly, we have
\begin{equation}\label{3.6.1}
l_{p\geq2}(T_1)=f_{E^-,(p,(c-1)_c)}\circ f_{C^+,(p,(c-2)_c)}\circ f_{E^-,(1,1_c)}\circ f_{C^+,(1,2_c)}(T_1),
\end{equation}
then $\Phi_1=[4p-2,4p+1]$;

Secondly, we have
\begin{equation}\label{3.6.2}
l_2(T_2)=(\sum\nolimits_{i=3}^{i=p}f_{C^+,(i,i-1)})|_{p\geq3}\circ f_{D^+,(1,2)}(T_2),
\end{equation}
then $\Phi_2=\{4p+4+2s|s=0,\cdots,p-2\}$;

Thirdly, we have
\begin{equation}\label{3.6.3}
l_3(T_3)=f_{A^+,(2,0_c)} \circ \sum\nolimits_{i=p}^{i=2}f_{C^-,(i,i-1)}\circ f_{B^-,(1_c,1)}(T_3),
\end{equation}
then $\Phi_3=\{4p+2,6p-1-2s|s=0,\cdots,p-1\}$.

Thus we obtain $\Phi_1\cup\Phi_2\cup\Phi_3=[4p-2,6p]$.

On this composite walk, there are two tuples of triangulations equipped with the same number of arrows, that is, $T_2$ and $f_{C^+,(1,2_c)}(T_1)$, $f_{E^+,(p,3_c)}(T_2)$ and $f_{A^-,(2,0_c)}\circ l_3(T_3)$. While $T_2$ and $f_{C^+,(1,2_c)}(T_1)$ are obviously not isotopy, $f_{E^+,(p,3_c)}(T_2)$ and $f_{A^-,(2,0_c)}\circ l_3(T_3)$ are mirror symmetry which are different vertices in $\mathbf{E^\circ(S,M)}$.

Hence $l_3\circ l_2\circ l_{p\geq2}(T_1)$ is a complete walk.

\textbf{8)} For the surface where $g=0,b=1,p\geq2,c\geq5$, $t_{min}=4p+c-6$ and $t_{max}=6p+\frac{3}{2}c-\frac{m}{2}-6$. Let $T_1$ be a triangulation equipped with $t_{min}$ arrows shown below. Then a walk staring from $T_1$ is $l_1(T_1)=l_{p=0}\circ l_{p\geq2}(T_1)$, so that $\Phi_1$ is $[4p+c-6,4p+\frac{3}{2}c-\frac{m}{2}-6]$.

We consider two local triangulations of $T_2=l_1(T_1)$ shown below. One is a rectangle containing the piece $\triangle_0$ in which the puncture labeled by $p$ lies, the other is a rectangle containing a family of pieces $\triangle_0$ with $p-1$ punctures. In these rectangles, the degree of point labeled by $k'$ or $k$ is at least 3 (the reason is consistent with the above proof in the surface of Case 3.4), so that we can do the flip of type $C$. Then the following walks in (\ref{3.7.1}) and (\ref{3.7.2}) are to flip in these rectangles and keep the other part unchanged.\\
\begin{figure}[H]
\centering
\includegraphics[width=12cm]{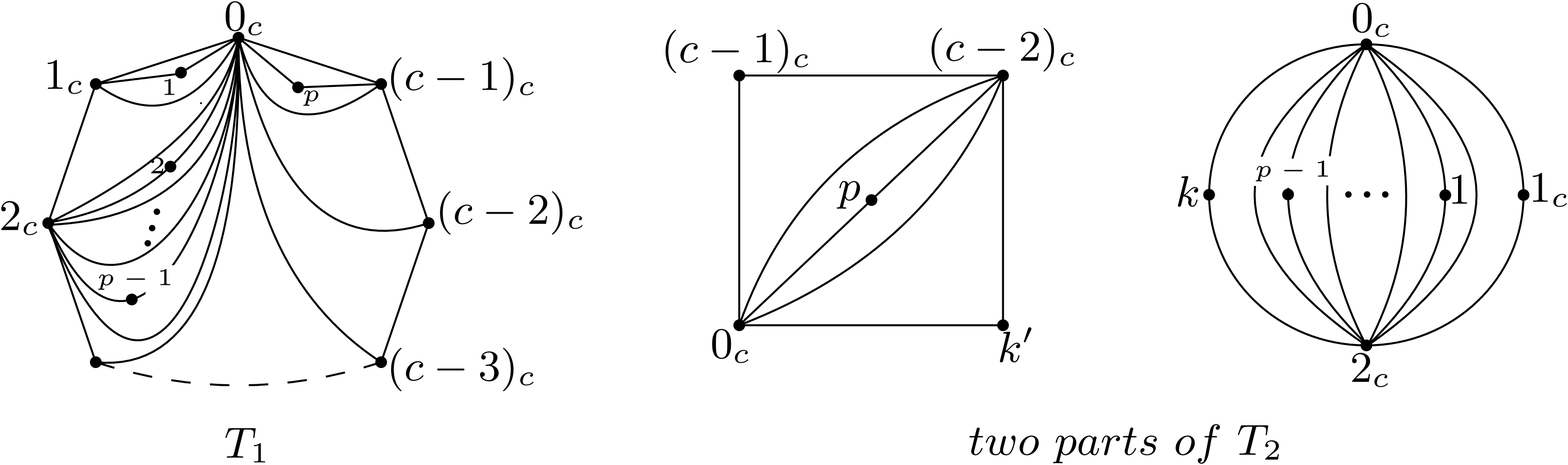}
\end{figure}
Firstly, we have
\begin{equation}\label{3.7.1}
l_2(T_2)=(\sum\nolimits_{i=p-2}^{i=1}f_{C^+,(i,i+1)})|_{p\geq3}\circ f_{C^+,(p-1,k)}\circ f_{C^+,(p,k')}(T_2),
\end{equation}
then $\Phi_2=\{4p+\frac{3}{2}c-\frac{m}{2}-6+2s|s=1,\cdots,p\}$;

Secondly, we have
\begin{equation}\label{3.7.2}
l_3(T_3)=f_{C^-,(p,k')} \circ (\sum\nolimits_{i=p-1}^{i=2}f_{C^-,(i,i-1)})|_{p\geq3}\circ f_{H,(p-1,k)}\circ f_{B^-,(1_c,1)}(T_3),
\end{equation}
where $T_3=l_2(T_2)$, then $\Phi_3=\{6p+\frac{3}{2}c-\frac{m}{2}-7-2s|s=0,\cdots,p-1\}$.

Thus we obtain $\Phi_1\cup\Phi_2\cup\Phi_3=[4p+c-6,6p+\frac{3}{2}c-\frac{m}{2}-6]$.

On this composite walk, there are two tuples of triangulations equipped with the same number of arrows, that is, $l_{p\geq2}(T_1)$ and $f_{C^+,(1,2_c)}(T_1)$, two triangulations related by the flip of type $H$. While they are obviously not isotopy.

Hence $l_3\circ l_2\circ l_1(T_1)$ is a complete walk.

\textbf{Case4:} For $\mathbf{S_4}$ where $g\geq1,b=0,p\geq1,c=0$, let $T_1$ be a triangulation obtained by gluing the block shown below where $p-1$ punctures lie and a number of pieces $\triangle_3$.
Then the following walks in (\ref{4.1}) and (\ref{4.2}) are to flip in that block and keep the other part unchanged.\\
\begin{figure}[H]
\centering
\includegraphics[width=3.5cm]{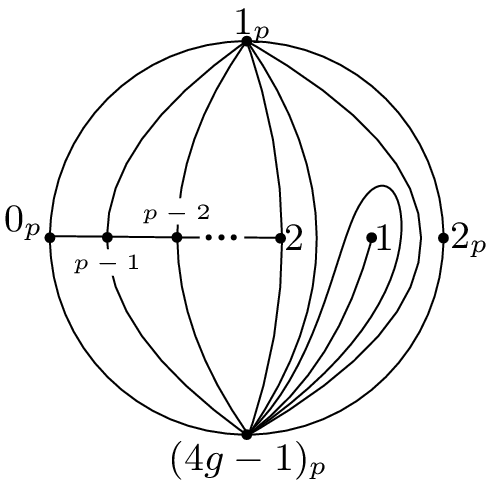}
\end{figure}
Firstly, we have
\begin{equation}\label{4.1}
l_1(T_1)=f_{A^-,(1,1_p)}\circ (f_{C^-,(p-1,0_p)})|_{p\geq3} \circ (\sum\nolimits_{i=2}^{i=p-2}f_{C^-,(i,i+1)})|_{p\geq4}(T_1),
\end{equation}
then $\Phi_1=\{2n-2p+2,2n-2p+3+2s|s=0,\cdots,p-2\}$;

Secondly, we have
\begin{equation}\label{4.2}
l_2(T_2)=(\sum\nolimits_{i=2}^{i=p-1}f_{C^+,(i,i-1)})|_{p\geq3}\circ f_{C^+,(1,2_p)}(T_2),
\end{equation}
where $T_2=l_1(T_1)$, then $\Phi_2=\{2n-2p+2+2s|s=1,\cdots,p-1\}$.

Thus we obtain $\Phi_1\cup\Phi_2=[2n-2p+2,2n]$. And any triangulation is equipped with different number of arrows on this composite walk.

Hence $l_2\circ l_1(T_1)$ is a complete walk.

\textbf{Case5:} For $\mathbf{S_5}$ where $g=0,b\geq2,p\geq0,c\geq2$ or $g\geq1,b\geq1,p\geq0,c\geq1$, in the triangulation equipped with $t_{min}$ arrows, there is no piece $\triangle_1$ and all the punctures lie in the interior of a family of $\triangle_0$.
In the triangulation equipped with $t_{max}$ arrows which is $2n-\frac{c+m}{2}$, the number of pieces $\triangle_1$ is $\frac{c-m}{2}$ and the degree of any puncture is at least 3.

Let $T_1$ be a triangulation equipped with $t_{max}$ arrows satisfying two conditions:
($A$) each piece $\triangle_1$ is glued to a piece $\triangle_3$ whose endpoints are not punctures, while two various pieces $\triangle_1$ can not be glued to the same piece $\triangle_3$;
($B$) when $p\geq2$, all the punctures lie in the block shown in the left of three graphs below; when $p=1$, the puncture lies in a piece $\triangle_3$ glued by a piece $\triangle_1$ and is of degree 3.



Such $T_1$ of $\mathbf{S_5}$ can be obtained by the following construction:
1) proceed labeling procedure on surface $\mathbf{S'_5}$;
2) on a boundary $j$ of $\mathbf{S'_5}$ with $c_j$ marked points, if $c_j$ is odd, link $i-2$ with $i$ for $i=2,4,\cdots,c_j-1$; otherwise, link $i-2$ with $i$ for $i=2,4,\cdots,c_j$ ($c_j$ is the point labeled by 0). Make sure that after linking, there is a piece $\triangle_1$ whose endpoints are $i-2,\ i-1$, and $i$.
3) choose a suitable point on the other borderline of $\mathbf{S'_5}$ and link it with two boundary points which are endpoints of the mutable side of a piece $\triangle_1$;
4) supplement arcs to get a triangulation of $\mathbf{S'_5}$;
5) if $p\geq2$, replace an arc of $\mathbf{S'_5}$ whose endpoints are on different borderlines with the block; if $p=1$, choose a piece $\triangle_3$ glued by a piece $\triangle_1$ and link it with all three endpoints of that $\triangle_3$.

There are $\frac{c-m}{2}$ pieces $\triangle_1$ in $T_1$. Because $T_1$ satisfies condition $A$, we can make $\frac{c-m}{2}$ times flips of type $B^-$. Then we get $l_1(T_1)$, so that $\Phi_1$ is $[2n-c,2n-\frac{c+m}{2}]$.

When $p=0$, $l_1(T_1)$ is a complete walk.

When $p=1$, we consider the local triangulation of $T_2=l_1(T_1)$ containing the puncture, shown in the middle of three graphs below. Then $f_{A^+,(p,v_1)}\circ f_{C^-,(p,v_4)}\circ f_{H,(p,v_2)}\circ l_1(T_1)$ is a complete walk.

When $p\geq2$, we consider the local triangulation of $T_2$ shown in the right of three graphs below containing the block with $p$ punctures.
In this rectangle, vertices labeled by $v_i(i=1,\cdots,4)$ are boundary points which may be the same, satisfying that $v_2$ and $v_4$ are not the junction points of two frozen sides and the degree of $v_i$ is at least 3.
Then the following walks in (\ref{5.1}) and (\ref{5.2}) are to flip in this rectangle and keep the other part unchanged.\\
\begin{figure}[H]
\centering
\includegraphics[width=10.5cm]{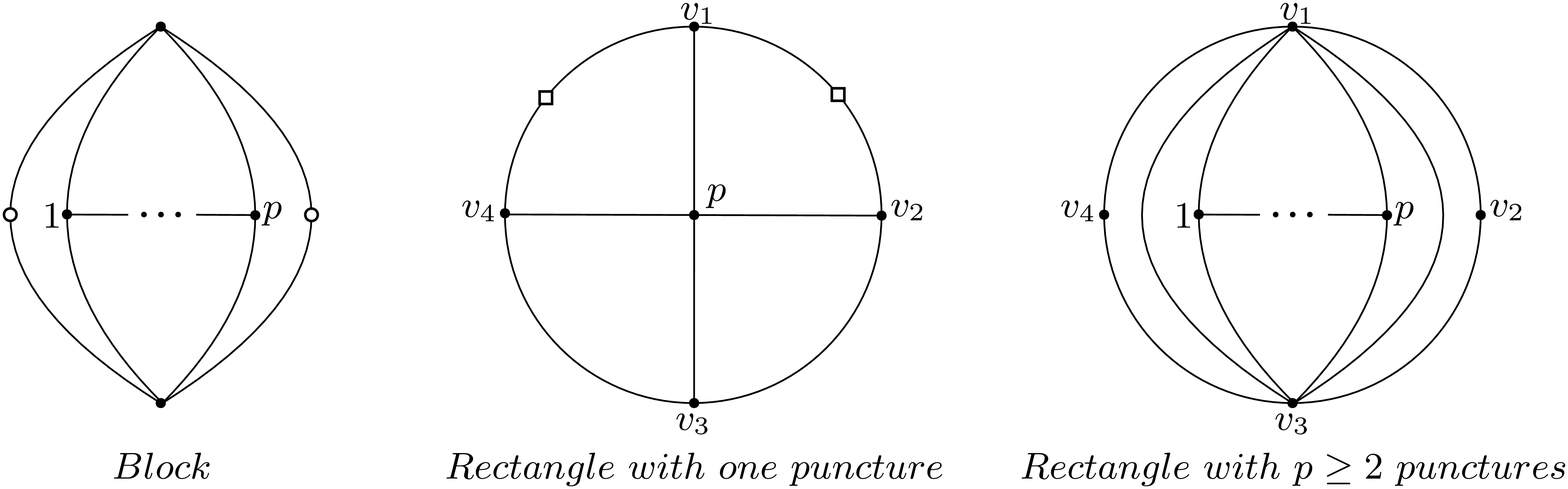}
\end{figure}
Firstly, we have
\begin{equation}\label{5.1}
l_2(T_2)=f_{C^-,(p,v_2)} \circ \sum\nolimits_{i=1}^{i=p-1}f_{C^-,(i,i+1)}\circ f_{H,(p,v_2)}(T_2),
\end{equation}
then $\Phi_2=\{2n-c-2s|s=1,\cdots,p\}$;

Secondly, we have
\begin{equation}\label{5.2}
l_3(T_3)=(\sum\nolimits_{i=2}^{i=p-1}f_{C^+,(i,i-1)})|_{p\geq3}\circ f_{C^+,(1,v_4)} \circ f_{A^+,(p,v_1)}(T_3),
\end{equation}
where $T_3=l_2(T_2)$, then $\Phi_3=\{2n-c-2p+1+2s|s=0,\cdots,p-1\}$.

Thus we obtain $\Phi_1\cup\Phi_2\cup\Phi_3=[2n-c-2p,2n-\frac{c+m}{2}]$. And on this composite walk, there are only two triangulations equipped with the same number of arrows, which are related by the flip of type $H$. While they are obviously not isotopy.

Hence $l_3\circ l_2\circ l_1(T_1)$ is a complete walk.
\end{proof}

\begin{Corollary}
For a surface $\mathbf{S}$ which is not the twice-punctured digon or the forth-punctured sphere, the distribution set for $\mathbf{S}$ with respect to $Q$ is continuous.
\label{surface}
\end{Corollary}

The proof of Lemma \ref{main} with respect to exchange cluster quivers of exceptional types is based on the following result.

\begin{Lemma}
Let $Q$ be an exchange cluster quiver of exceptional type $E_6,\ E_7,\ E_8,\ \widetilde{E}_6,\ \widetilde{E}_7$, $\widetilde{E}_8,\ E_6^{(1,1)},\ E_7^{(1,1)}$ or $E_8^{(1,1)}$, then the distribution set for $Mut[Q]$ is continuous.
\label{exp}
\end{Lemma}
\begin{proof}
Mills in \cite{M} used Sage and obtained all quivers in each exceptional finite mutation equivalence classes except for $E_6,E_7,E_8$ and $X_7$, and on webpage \cite{Mweb} he provided the corresponding skew-symmetric matrixes of quivers. Thanks to his work, we compute the numbers of arrows of quivers in each exceptional mutation equivalence class except for $E_6,E_7,E_8$ and $X_7$ by Python and Matlab.
For each type, there are quivers with the same number of arrows and we show one quiver for each number.
We denote $Q^{(i)}$ as the $i$th quiver numbered by Mills in the mutation equivalence class, and $W_K$ as the distribution set for $Mut[Q_K]$ where $Q_K$ is of exceptional type $K$.
\begin{align}\begin{split}
\ \ W_{\widetilde{E}_6}
  &=\{|Q^{(0)}_1|,|Q^{(3)}_1|,|Q^{(4)}_1|,|Q^{(20)}_1|,|Q^{(19)}_1|,|Q^{(60)}_1|,|Q^{(40)}_1|\}\\
  &=\{6,7,8,9,10,11,12\}\nonumber\\
\ \ W_{\widetilde{E}_7}
  &=\{|Q^{(0)}_1|,|Q^{(3)}_1|,|Q^{(8)}_1|,|Q^{(33)}_1|,|Q^{(95)}_1|,|Q^{(220)}_1|,|Q^{(568)}_1|,|Q^{(907)}_1|\}\\
  &=\{7,8,9,10,11,12,13,14\}\nonumber\\
\ \ W_{\widetilde{E}_8}
  &=\{|Q^{(0)}_1|,|Q^{(2)}_1|,|Q^{(18)}_1|,|Q^{(54)}_1|,|Q^{(161)}_1|,|Q^{(303)}_1|,|Q^{(1691)}_1|,|Q^{(4500)}_1|,|Q^{(6669)}_1|\}\\
  &=\{8,9,10,11,12,13,14,15,16\}\nonumber\\
\ \ W_{E_6^{(1,1)}}
  &=\{|Q^{(32)}_1|,|Q^{(6)}_1|,|Q^{(0)}_1|,|Q^{(24)}_1|,|Q^{(11)}_1|\}\\
  &=\{9,10,11,12,13\}\nonumber\\
\ \ W_{E_7^{(1,1)}}
  &=\{|Q^{(0)}_1|,|Q^{(2)}_1|,|Q^{(1)}_1|,|Q^{(10)}_1|,|Q^{(22)}_1|,|Q^{(43)}_1|,|Q^{(42)}_1|,|Q^{(374)}_1|\}\\
  &=\{9,10,11,12,13,14,15,16\}\nonumber \\
\ \ W_{E_8^{(1,1)}}
&=\{|Q^{(0)}_1|,|Q^{(3)}_1|,|Q^{(1)}_1|,|Q^{(11)}_1|,|Q^{(28)}_1|,|Q^{(102)}_1|,|Q^{(856)}_1|,|Q^{(2237)}_1|,|Q^{(5737)}_1|\}\\
  &=\{10,11,12,13,14,15,16,17,18\}\nonumber\\
\ \ W_{X_6}
  &=\{|Q^{(0)}_1|,|Q^{(1)}_1|\}\\
  &=\{9,11\}\nonumber
\end{split}\end{align}

For $E_6,E_7,E_8$ and $X_7$, our main tool is Keller's quiver mutation in Java \cite{Java}.
There are 21, 112, 391 and 2 quivers in the mutation equivalence class(up to isomorphism) of type $E_6,E_7,E_8$ and $X_7$ respectively. Then we compute the numbers of arrows of all quivers in each exceptional mutation equivalence class and have the following result:
$$W_{E_6}=\{5,6,7,8,9\},\ W_{E_6}=\{5,6,7,8,9\},\ W_{E_8}=\{7,8,9,10,11,12,13\},\ W_{X_6}=\{12,15\}.$$
\end{proof}
Based on the above discussion, we have the following theorem.
\begin{Theorem}
Let $Q$ be an exchange cluster quiver of finite mutation type, then the distribution set for $Mut[Q]$ is continuous unless:
\begin{enumerate}[(i)]
  \item $Q$ is of type $X_6$ or $X_7$;
  \item $Q$ arises from the twice-punctured digon $\mathbf{S_{2p-dig}}$ or the forth-punctured sphere $\mathbf{S_{4p-sphere}}$.
\end{enumerate}
More specifically, we have
$$W_{\mathbf{S_{2p-dig}}}=\{4,6,7,8\},\ W_{\mathbf{S_{4p-sphere}}}=\{8,9,10,12\},\ W_{X_6}=\{9,11\},\ W_{X_7}=\{12,15\}.$$
\label{1}
\end{Theorem}
\begin{proof}
  Combining Lemma \ref{bulianxu}, Corollary \ref{surface}, Lemma \ref{exp} and the fact that for the exchange cluster quiver $Q$ where $|Q_0|=2,\ |Q_1|=n$, the number of arrows for $Mut[Q]$ obviously presents a continuous distribution. Then the conclusion is complete proved.
\end{proof}

\section{The proof of Theorem \ref{main} with respect to extended exchange cluster quivers}
As we have said in the introduction, according to the calculation formulas of the numbers of arrows of $Q$ and $\tilde{Q}$ arising from surfaces (see Corollary \ref{calculation} and Proposition \ref{calculate-exQ}), the conditions for the numbers of arrows to be extremum in $Mut[Q]$ and $Mut[\tilde{Q}]$ are different, as well as so do the triangulations associated with quivers in $Mut[Q]$ and $Mut[\tilde{Q}]$ respectively. Therefore, we need to divide our discussion into two cases, namely with respect to $Q$ and $\tilde{Q}$ respectively. In particular, when $b=0$ for the surface, the conclusions for the two cases are the same.

In this section, we give the proof of Theorem \ref{main} with respect to extended exchange cluster quivers.

\subsection{Two lemmas}
We first discuss the distribution sets $\tilde{W}_\mathbf{S}$ for several special surfaces.
\begin{Lemma}
The distribution sets for $\mathbf{S_{2p-mon}}$ and $\mathbf{S_{2p-mon}}$ with respect to $\tilde{Q}$ are continuous, where $\mathbf{S_{2p-mon}}$ is the twice-punctured monogon and $\mathbf{S_{1p-dig}}$ is the once-punctured digon.
\label{ex-78}
\end{Lemma}
\begin{proof}
For $\mathbf{S_{2p-mon}}$, there are triangulations equipped with the same number arrows. For each number, we list one triangulation in Figure \ref{2p-mon}. Then $\tilde{W}_{\mathbf{S_{2p-mon}}}=\{6, 7, 8\}$.

It is easy to see that the distribution set $\tilde{W}_{\mathbf{S_{1p-dig}}}=\{4\}$.
\end{proof}
\begin{Lemma}
The distribution sets for $\mathbf{S_{1p-tri}}$ and $\mathbf{S_{4p-sphere}}$ with respect to $\tilde{Q}$ are continuous, where $\mathbf{S_{1p-tri}}$ is the once-punctured triangle and $\mathbf{S_{4p-sphere}}$ is the forth-punctured sphere. More specifically,
$$\tilde{W}_{\mathbf{S_{1p-tri}}}=\{6,7,9\},\ \tilde{W}_{\mathbf{S_{4p-sphere}}}=\{8,9,10,12\}.$$
\label{ex-bulianxu}
\end{Lemma}
\begin{proof}
For $\mathbf{S_{1p-tri}}$, triangulations are isotopy to one of the graphs shown in the Figure \ref{1p-triangle}. There is no extended exchange cluster quiver with 8 arrows arising from $\mathbf{S_{1p-tri}}$.
\begin{figure}[H]
\centering
\includegraphics[width=9cm]{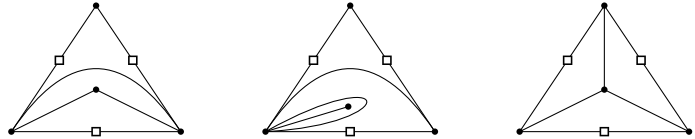}
\caption{Triangulations $T$ of $\mathbf{S_{1p-tri}}$ (up to isotopy) equipped with 6,7,9 arrows from left to right}
\label{1p-triangle}
\end{figure}

For $\mathbf{S_{4p-sphere}}$, because there is no boundary, we have $Mut[\tilde{Q}_\mathbf{S}]=Mut[Q_\mathbf{S}]$, so that $\tilde{W}_{\mathbf{S_{4p-sphere}}}=W_{\mathbf{S_{4p-sphere}}}$.
\end{proof}

By the relation between triangulations and associated extended exchange cluster quivers, we give the number of arrows corresponding to each piece shown in Table \ref{ex-num-ori-tri}, where the negative integer means that when there is a piece $\triangle_0$, the number of arrows should be decreased by 2.
\begin{table}[H]
\centering
\includegraphics[width=16cm]{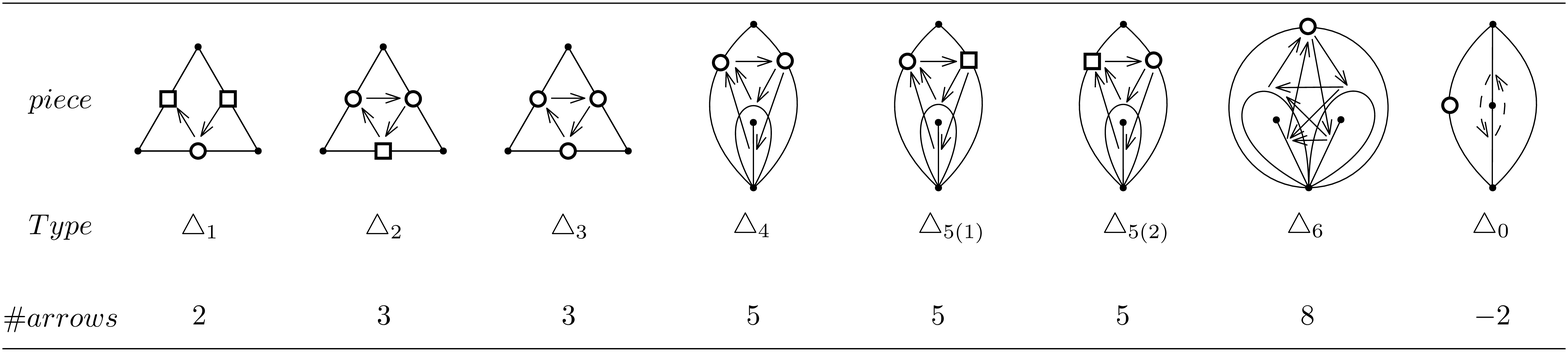}
\caption{The number of arrows for each piece with respect to $\tilde{Q}$}
\label{ex-num-ori-tri}
\end{table}

\begin{Proposition}
For a surface $\mathbf{S}$ with $n$ arcs and $c$ marked points on the boundary, which is not one of the twice-punctured monogon, the once-punctured digon and the forth-punctured sphere, let $t_i$ be the number of the piece $\triangle_i$ in a triangulation $T$ of $\mathbf{S}$. Then the \textbf{calculation formula for the number of arrows of $\tilde{Q}_T$} is $\tilde{t}=2n+c-t_1-t_4-t_5-t_6-2t_0$.
\label{calculate-exQ}
\end{Proposition}
\begin{proof}
According to the relation(see Table \ref{ex-num-ori-tri}) between the number of arrows of $\tilde{Q}_T$ and the number of arrows for different kinds of pieces in $T$, we have
$$\tilde{t}=2t_1+3(t_2+t_3)+5(t_4+t_5)+8t_6-2t_0.$$
Similar to the previous discussion in Corollary \ref{calculation}, we have
$$c=2t_1+t_2+t_5,
\ n-2(t_4+t_5)-4t_6=(t_1+2t_2+3t_3+2t_4+t_5+t_6)/2.$$
Therefore, we have the calculation formula for the number of arrows of $\tilde{Q}_T$.
\end{proof}

\subsection{The maximum and minimum numbers of arrows of $\tilde{Q}$ arising from surfaces}
Let $\tilde{t}_{max}$ and $\tilde{t}_{min}$ be the maximum and minimum numbers of the distribution set $\tilde{W}_\mathbf{S}$ for a surface $\mathbf{S}$ respectively. We first discuss the optimal condition for the maximum and minimum.

From the calculation formula for the number of arrows of $\tilde{Q}_T$, we can see that the optimal condition of $\tilde{t}_{max}$ arrows is that $t_i=0\ (i=0,1,4,5,6)$. The equivalent condition of that $t_1=0$ is that the degree of each boundary point is at least 3, denoted by \emph{Max1$'$}. Because the degree of the inner puncture of pieces $\triangle_i\ (i=0,4,5,6)$ is 2, thus the equivalent condition of that $t_i=0\ (i=0,4,5,6)$ is that the degree of each puncture is at least 3, denoted by \emph{Max2$'$}.

Thus the optimal situation for the triangulation equipped with $\tilde{t}_{max}$ arrows is that \emph{Max1$'$} and \emph{Max2$'$} hold simultaneously.
\begin{Lemma}
  For a surface where $g=0,b=1$, there is a triangulation satisfying condition \emph{Max1$'$} and \emph{Max2$'$} simultaneously if and only if $p\neq0$ and $p+c\geq4$.
  \label{Max}
\end{Lemma}
\begin{proof}
  ($\Rightarrow$) If there is no puncture, there will be at least two boundary points of degree 2 in the triangulation of a surface where $c\geq4$ (recall that we do not allow unpunctured monogon , digon or triangle), which is a contradiction to \emph{Max1$'$}. If $p+c\leq3$, it is easy to see that \emph{Max1$'$} and \emph{Max2$'$} cannot hold simultaneously, which is a contradiction.

  ($\Leftarrow$) For punctures in the un-triangulated surface, link them together into a $p-$polygon when $p\geq3$ or a line when $p=2$, then the degree of any puncture is 0 if $p=1$, 1 if $p=2$, or 2 if $p\geq3$.

  (i)If $p=1$, then $c\geq3$, link each boundary point to the puncture, then the degree of any boundary point is 3 and the degree of the puncture is $c$.

  (ii)If $p=2$, then $c\geq2$, link two boundary points to each puncture, then the degree of any punctures is at least 3; link boundary points of degree 2 to suitable punctures, then the degree of any boundary point is at least 3.

  (iii)If $p\geq3$, then $c\geq1$, link a boundary point to punctures on the $p-$polygon, then the degree of any punctures is at least 3; link boundary points of degree 2 to suitable punctures, then the degree of any boundary point is at least 3.

  Thus \emph{Max1$'$} and \emph{Max2$'$} hold simultaneously.
\end{proof}

From the calculation formula for the number of arrows of $\tilde{Q}_T$, we can see that the optimal condition of $\tilde{t}_{min}$ is that $t_1=\frac{c-m}{2}$ and $t_0=p$, where $m$ is the number of boundary components with odd boundary points.
For a surface with boundaries, $t_1=\frac{c-m}{2}$ if and only if there are $\frac{c-m}{2}$ boundary points of degree 2.
Then a sufficient condition of $t_1=\frac{c-m}{2}$ is that the degree of each boundary point labeled odd is 2 in the labeled surface, denoted by \emph{Min1$'$.}

The equivalent condition of that there is a triangulation where $t_0=p$ is that there are two boundary points labeled the same on the different borderline or labeled different, denoted by \emph{Min2$'$}.

Thus the optimal situation for the triangulation equipped with $\tilde{t}_{min}$ arrows is that \emph{Min1$'$} and \emph{Min2$'$} hold simultaneously. It is easy to see that a sufficient condition of that \emph{Min1$'$} and \emph{Min2$'$} hold simultaneously is that after proceeding labeling procedure, there are two points labeled by even number which are labeled the same on different borderlines, or labeled different.

\begin{Proposition}
  The maximum number $\tilde{t}_{max}$ and minimum number $\tilde{t}_{min}$ of the distribution set $\tilde{W}_\mathbf{S}$ for a surface $\mathbf{S}$ are given respectively in the following table, where $m$ is the number of boundary components with odd boundary points, $c$ is the number of boundary points, $p$ is the number of punctures and $n$ is the number of arcs.
\begin{table}[H]
\centering
\includegraphics[width=16cm]{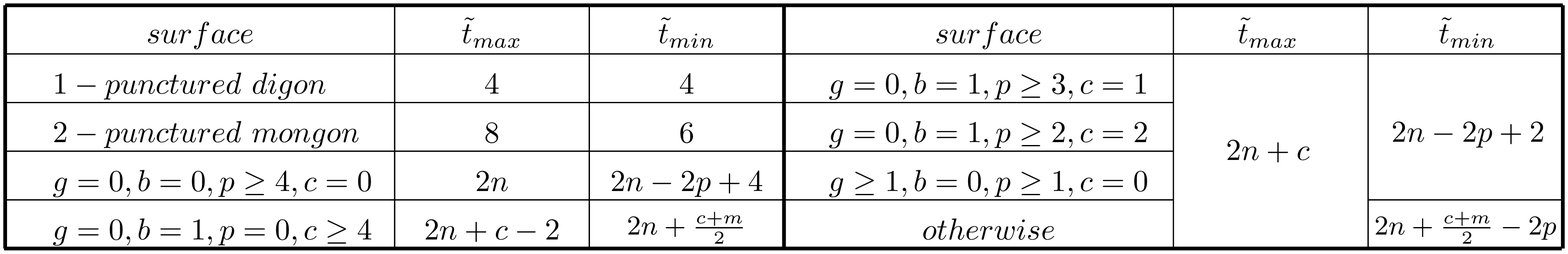}
\caption{The $\tilde{t}_{max}$ and $\tilde{t}_{min}$ of $\tilde{W}_\mathbf{S}$}
\label{type-t}
\end{table}
\label{num-check-t}
\end{Proposition}
\begin{proof}
For the surface $\mathbf{S}$ where $b=0$, we have $Mut[\tilde{Q}_\mathbf{S}]=Mut[Q_\mathbf{S}]$, so that $\tilde{W}_{\mathbf{S}}=W_{\mathbf{S}}$, that is, $\tilde{t}_{max}=t_{max}$ and $\tilde{t}_{min}=t_{min}$, where $t_{max}$ and $t_{min}$ are the maximum and minimum numbers in $W_{\mathbf{S}}$.

For the surface $\mathbf{S}$ where $b\neq0$, we discuss on several cases.
\vspace{1 ex}

\textbf{Case1:} $\mathbf{S_1}$ where $g=0,b=1,p=0,c\geq4$.

For $p=0$, then $\tilde{t}=2n+c-t_1$, and it's easy to see that \emph{Max2$'$}, \emph{Min1$'$} and \emph{Min2$'$} hold, but \emph{Max1$'$} does not. Obviously, the minimum $t_1$ is 2 and the maximum $t_1$ is $\frac{c-m}{2}$.

Then by proceeding labeling procedure and linking procedure, we can obtain a triangulation of $\mathbf{S_1}$ equipped with $\tilde{t}_{max}$ arrows which is $2n+c-2$.
And by Minimum Construction $\uppercase\expandafter{\romannumeral1}'$, we can obtain a triangulation of $\mathbf{S_1}$ equipped with $\tilde{t}_{min}$ arrows which is $2n+\frac{c+m}{2}$.

\emph{Minimum Construction $\uppercase\expandafter{\romannumeral1}'$}
\par $\bullet$ \emph{Proceeding labeling procedure.}
\par $\bullet$ \emph{On a boundary $j$ with $c_j$ marked points, if $c_j$ is odd, link $i-2$ with $i$ for $i=2,4,\cdots,c_j-1$; otherwise, link $i-2$ with $i$ for $i=2,4,\cdots,c_j$ ($c_j$ is the point labeled by 0). Make sure that after linking, there is a piece $\triangle_1$ whose endpoints are $i-2,\ i-1$, and $i$.}
\par $\bullet$ \emph{delete one side of any contractible digon that could appear after the above steps.}
\vspace{1 ex}

\textbf{Case2:} $\mathbf{S_2}$ where $g=0,b=1,p\geq3,c=1$.

By Lemma \ref{Max}, \emph{Max1$'$} and \emph{Max2$'$} hold simultaneously.
Note that \emph{Max1$'$} and \emph{Max2$'$} also hold simultaneously for a triangulation $T_{max}$ equipped with $t_{max}$ arrows.
Thus the triangulation of $\mathbf{S_2}$ equipped with $\tilde{t}_{max}$ arrows is exactly $T_{max}$, and $\tilde{t}_{max}=2n+c$.

Obviously, \emph{Min1$'$} holds but \emph{Min2$'$} does not. Then the maximum pieces $\triangle_0$ is $p-1$. Note that for a triangulation $T_{min}$ equipped with $t_{min}$ arrows, \emph{Min1$'$} holds and $(t_0)_{max}=p-1$.
Thus the triangulation of $\mathbf{S_2}$ equipped with $\tilde{t}_{min}$ arrows is exactly $T_{min}$, and $\tilde{t}_{min}=2n-2p+2$.
\vspace{1 ex}

\textbf{Case3:} $\mathbf{S_3}$ where $g=0,b=1,p\geq2,c=2$.

By Lemma \ref{Max}, \emph{Max1$'$} and \emph{Max2$'$} hold simultaneously. Then by Maximum Construction $\uppercase\expandafter{\romannumeral3}'$, we can obtain a triangulation of $\mathbf{S_3}$ equipped with $\tilde{t}_{max}$ arrows which is $2n+c$.

\emph{Maximum Construction $\uppercase\expandafter{\romannumeral3}'$}
\par $\bullet$ \emph{Proceed the labeling procedure.}
\par $\bullet$ \emph{Link all punctures into a line.}
\par $\bullet$ \emph{Link two labeled points respectively with all punctures on the line.}

For $\mathbf{S_3}$, $(t_1)_{max}=1$ and $(t_0)_{max}=p$. Because $t_0=p$ and $c=2$, the triangulation of $\mathbf{S_3}$ with $t_1=1,t_0=p$ can only consist of a piece $\triangle_1$ and a family of pieces $\triangle_0$.
While there is a contractible loop in the graph obtained by gluing a piece $\triangle_1$ and a family of pieces $\triangle_0$ together, which is a contradiction. Thus \emph{Min1$'$} and \emph{Min2$'$} do not hold simultaneously.
The triangulation with $t_1=1,t_4=1,t_0=p-1$ and the one with $t_1=0,t_4=0,t_0=p$ are equipped with the same number of arrows, so we may create the latter one for minimum case.
Thus by Minimum Construction $\uppercase\expandafter{\romannumeral3}'$, we can obtain a triangulation of $\mathbf{S_3}$ equipped with $\tilde{t}_{min}$ arrows which is $2n-2p+2$.

\emph{Minimum Construction $\uppercase\expandafter{\romannumeral3}'$}
\par $\bullet$ \emph{Construct a family of pieces $\triangle_0$ with $p$ punctures}
\par $\bullet$ \emph{Replace a side on the boundary with a family of pieces $\triangle_0$.}
\vspace{1 ex}

\textbf{Case4:} $\mathbf{S_4}$ where $g=0,b=1,p\geq1,c\geq3$.

By Lemma \ref{Max}, \emph{Max1$'$} and \emph{Max2$'$} hold simultaneously. Then by Maximum Construction $\uppercase\expandafter{\romannumeral4}'$, we can obtain a triangulation of $\mathbf{S_4}$ equipped with $\tilde{t}_{max}$ arrows which is $2n+c$.

\emph{Maximum Construction $\uppercase\expandafter{\romannumeral4}'$}
\par $\bullet$ \emph{Proceed Maximum Construction $\uppercase\expandafter{\romannumeral3}'$.}
\par $\bullet$ \emph{Link the labeled points of degree 2 with suitable punctures of degree 3 on the line.}

Obviously, after proceeding labeling procedure, $\mathbf{S_4}$ satisfies the sufficient condition of that \emph{Min1$'$} and \emph{Min2$'$} hold simultaneously. Thus by Minimum Construction $\uppercase\expandafter{\romannumeral4}'$, we can obtain a triangulation of $\mathbf{S_4}$ equipped with $\tilde{t}_{min}$ arrows which is $2n+\frac{c+m}{2}-2p$.

\emph{Minimum Construction $\uppercase\expandafter{\romannumeral4}'$}
\par $\bullet$ \emph{Proceed Minimum Construction $\uppercase\expandafter{\romannumeral1}'$}
\par $\bullet$ \emph{Construct a family of pieces $\triangle_0$ by unconnected punctures}
\par $\bullet$ \emph{Replace an edge whose endpoints are labeled by even numbers by a family of pieces $\triangle_0$.}
\vspace{1 ex}

\textbf{Case5:} $\mathbf{S_5}$ where $g=0,b\geq2,p\geq0,c\geq2$ or $g\geq1,b\geq1,p\geq0,c\geq1$.

\emph{Max1$'$} and \emph{Max2$'$} hold simultaneously for $\mathbf{S_5}$.
Because there are at least 2 borderlines in the snipped surface $\mathbf{S'_5}$, so that there is no piece $\triangle_1$ after proceeding linking procedure, that is, \emph{Max1$'$} holds. Besides, if $p=0$, \emph{Max2$'$} holds. Otherwise, a triangulation of $\mathbf{S_5}$ can be obtained from a triangulation satisfying \emph{Max1$'$} of $\mathbf{S'_5}$ by replacing a piece $\triangle_2$ or $\triangle_3$ with a $p$th-punctured triangle.

For a triangle with $p$ punctures, \emph{Max1$'$} and \emph{Max2$'$} hold simultaneously, so does $\mathbf{S_5}$.
Then by Maximum Construction $\uppercase\expandafter{\romannumeral5}'$, we can obtain a triangulation of $\mathbf{S_5}$ equipped with $\tilde{t}_{max}$ arrows which is $2n+c$.

\emph{Maximum Construction $\uppercase\expandafter{\romannumeral5}'$}
\par $\bullet$ \emph{If $g\geq1$, snip the surface from a boundary point into a plane graph whose underlying graph is a 4g-sided polygon with $b$ holes.}
\par $\bullet$ \emph{Proceed labeling procedure and linking procedure}
\par $\bullet$ \emph{Proceed Maximum Construction $\uppercase\expandafter{\romannumeral3}'$.}
\par $\bullet$ \emph{Supplement arcs to get a complete triangulation.}

Obviously, after proceeding labeling procedure, $\mathbf{S_5}$ satisfies the sufficient condition of that \emph{Min1$'$} and \emph{Min2$'$} hold simultaneously. Thus by Minimum Construction $\uppercase\expandafter{\romannumeral5}'$, we can obtain a triangulation of $\mathbf{S_5}$ equipped with $\tilde{t}_{min}$ arrows which is $2n+\frac{c+m}{2}-2p$.

\emph{Minimum Construction $\uppercase\expandafter{\romannumeral5}'$}
\par $\bullet$ \emph{If $g\geq1$, snip the surface from a boundary point into a plane graph whose underlying graph is a 4g-sided polygon with $b$ holes.}
\par $\bullet$ \emph{Proceed Minimum Construction $\uppercase\expandafter{\romannumeral4}'$}
\par $\bullet$ \emph{Supplement arcs to get a complete triangulation.}
\end{proof}

\subsection{The existence of an extended complete walk}
The proof of Theorem \ref{main} with respect to extended exchange cluster quivers is based on the following Lemma. When considering the arrows between frozen vertices and mutable vertices, two types of flips $f_B$ and $f_E$ in Table \ref{walk-f} are changed to $f_{\tilde{B}}$ and $f_{\tilde{E}}$ shown in below, while the other types of flips keep unchanged.
\begin{figure}[H]
\centering
\includegraphics[width=16cm]{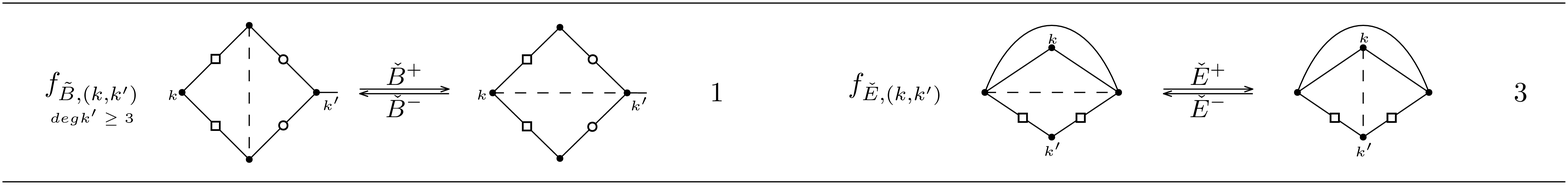}
\end{figure}
\begin{Lemma}
There is an extended complete walk in graph $\mathbf{E^\circ(S,M)}$ of a surface which is not an once-punctured triangle or a forth-punctured sphere.
\label{perfect2}
\end{Lemma}
\begin{proof}
For a twice-punctured monogon, an extended complete walk is $T_1-T_2-T_3$, where $T_i$ is the triangulation shown in \ref{2p-mon}. For a once-punctured digon, an extended complete walk is a vertex in $\mathbf{E^\circ(S,M)}$.

For the surface $\mathbf{S}$ where $b=0$, we have $Mut[\tilde{Q}_\mathbf{S}]=Mut[Q_\mathbf{S}]$, so that an extended complete walk is the same as complete walk.

For the surface $\mathbf{S}$ where $b\neq0$, we discuss on several cases.
\vspace{1 ex}

\textbf{Case1:} For $\mathbf{S_1}$ where $g=0,b=1,p=0,c\geq4$, let $T$ be the triangulation obtained by processing labeling procedure and linking procedure for $\mathbf{S_1}$, it can translate into the triangulation equipped with $\tilde{t}_{min}$ arrows by $\frac{c-m}{2}-2$ times flips of type $\tilde{B}^-$. Thus $l(T_1)=(\sum\nolimits_{i=1}^{i={\frac{c-m}{2}}-2}f_{\tilde{B}^-,((2i+1)_c,0_c)})|_{c\geq6}(T)$ is an extended complete walk for $\mathbf{S_1}$.
\vspace{1 ex}

\textbf{Case2:} For $\mathbf{S_2}$ where $g=0,b=1,p\geq3,c=1$, note that triangulations of $\mathbf{S_2}$ equipped with $\tilde{t}_{max}$ arrows and $\tilde{t}_{min}$ arrows respectively are the same as that with $t_{max}$ arrows and $t_{min}$ arrows. And on a complete walk given in Case2 in the proof of Lemma \ref{perfect1}, the set of the numbers of arrows for $\tilde{Q}_T$ associated with triangulations on that walk is exactly $[4p-2,6p-3]\cap\mathbb{N}$, so that an extended complete walk for $\mathbf{S_2}$ is the same as the complete walk.
\vspace{1 ex}

\textbf{Case3:} For $\mathbf{S_3}$ where $g=0,b=1,p\geq2,c=2$, let $T_1$ be a triangulation shown below. Then two triangulations $T_2=l_1(T_1)$ and $T_3=l_2(T_2)$ shown below will appear on the following walks in (\ref{3'1}), (\ref{3'2}), (\ref{3'3}).
\begin{figure}[H]
\centering
\includegraphics[width=10cm]{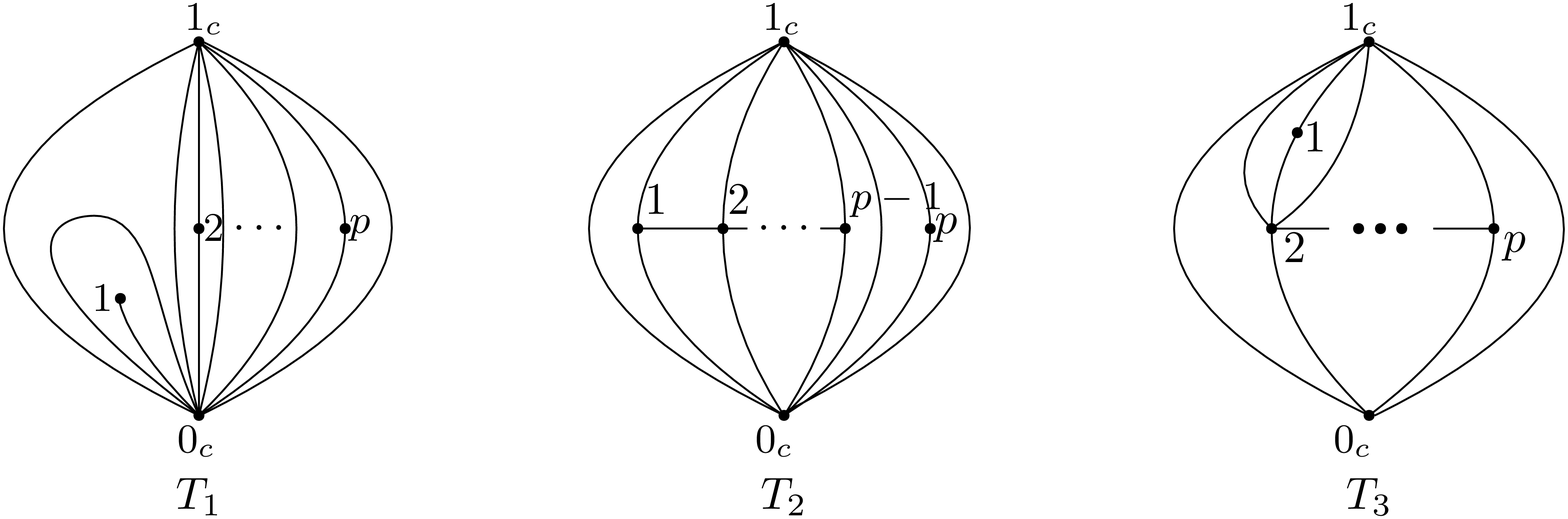}
\end{figure}
Firstly, we have
\begin{equation}\label{3'1}
l_1(T_1)=(\sum\nolimits_{i=3}^{i=p-1}f_{C^+,(i,i-1)})|_{p\geq4}\circ f_{D^+,(1,2)}\circ f_{A^-,(1,1_c)}(T_1),
\end{equation}
then, $\Phi_1=\{4p,4p+1,4p+4+2s|s=0,\cdots,p-3\}$;

Secondly, we have
\begin{equation}\label{3'2}
l_2(T_2)=f_{C^-,(1,0_c)}\circ (f_{C^+,(p,p-1)})|_{p\geq3}(T_2),
\end{equation}
then, $\Phi_2=\{6p-2,6p\}$;

Thirdly, we have
\begin{equation}\label{3'3}
l_3(T_3)=(f_{A^-,(1,1_c)})|_{p\geq3}\circ (\sum\nolimits_{i=p}^{i=3}f_{C^-,(i,i-1)})|_{p\geq3}\circ f_{A^+,(1,1_c)}(T_3),
\end{equation}
then, $\Phi_3=\{4p+2,4p+3+2s|s=0,\cdots,p-2\}$.

Thus we obtain $\Phi_1\cup\Phi_2\cup\Phi_3=[4p,6p]$.

It is easy to see that on this composite walk, there are only two triangulations $T_2$ and $T_3$ equipped with the same number of arrows which is $6p-2$. While they are obviously not isotopy.

Hence $l_3\circ l_2\circ l_1(T_1)$ is an extended complete walk for $\mathbf{S_3}$.
\vspace{1 ex}

\textbf{Case4:} For $\mathbf{S_4}$ where $g=0,b=1,p\geq2,c=3$, let $T_1$ be a triangulation shown below, then we have two walks in (\ref{4'1}) and (\ref{4'2}):
\begin{figure}[H]
\centering
\includegraphics[width=2.3cm]{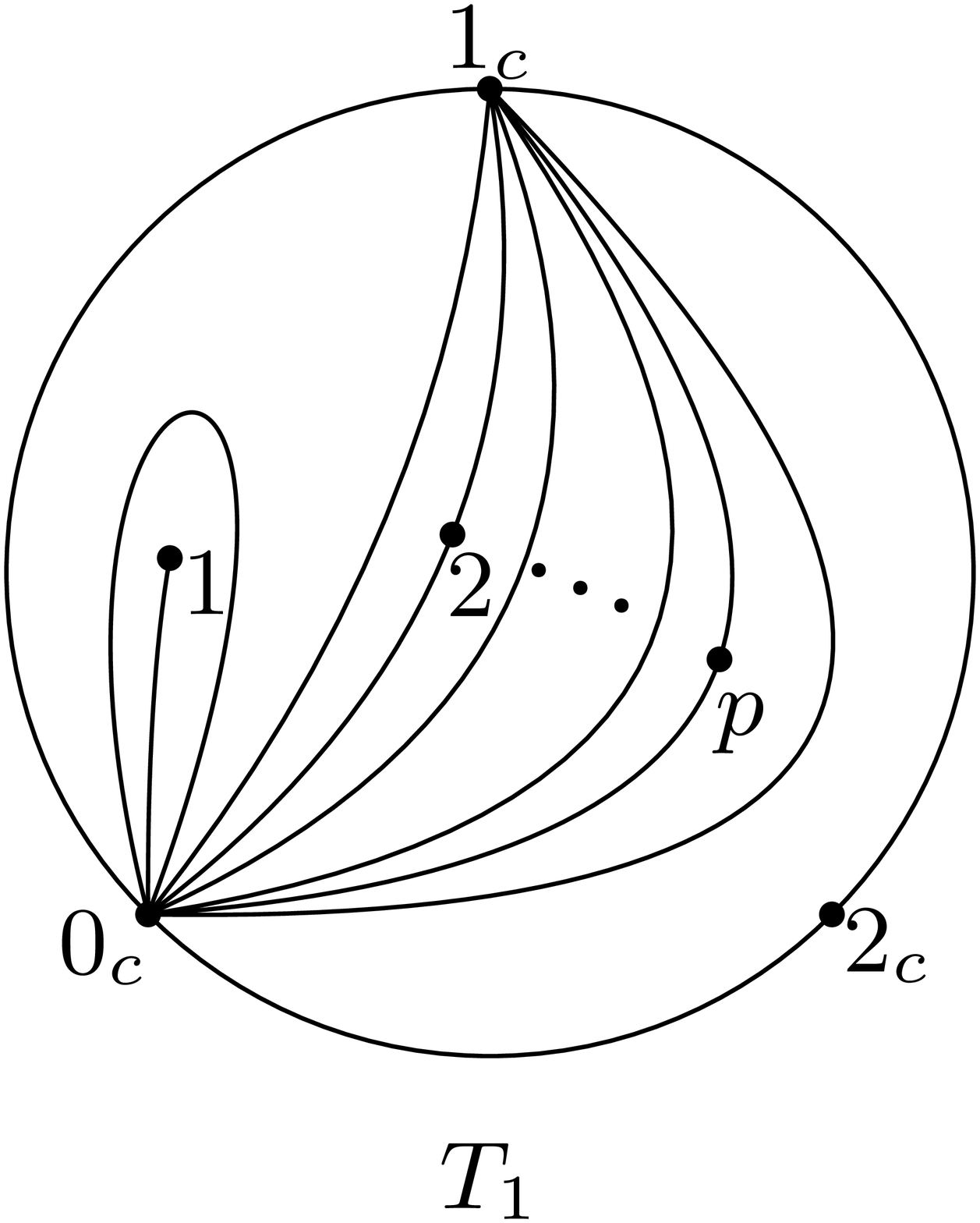}
\end{figure}
Firstly, we have
\begin{equation}\label{4'1}
l_1(T_1)=(\sum\nolimits_{i=p-1}^{i=1}f_{C^+,(i,i+1)})|_{p\geq4}\circ f_{\tilde{E}^+,(p,2_c)}\circ f_{A^-,(1,1_c)}(T_1),
\end{equation}
then $\Phi_1=\{4p+2,4p+3+2s|s=0,\cdots,p\}$;

Secondly, we have
\begin{equation}\label{4'2}
l_2(T_2)=f_{C^-,(1,0_c)}\circ (\sum\nolimits_{i=p}^{i=3}f_{C^-,(i,i-1)})|_{p\geq3}\circ f_{\tilde{B}^-,(2_c,p)}(T_2),
\end{equation}
where $T_2=l_1(T_1)$, then $\Phi_2=\{4p+4+2s|s=0,\cdots,p-1\}$.

Thus we obtain $\Phi_1\cup\Phi_2=[4p+2,6p+3]$. And on this composite walk, any triangulation is equipped with different number of arrows.

Hence $l_2\circ l_1(T_1)$ is an extended complete walk for $\mathbf{S_4}$.
\vspace{1 ex}

\textbf{Case5:} For $\mathbf{S_5}$ where $g=0,b=1,p\geq1,c=4$, let $T_1$ be a triangulation shown below. Then two triangulations $T_2=l_1(T_1)$ and $T_3=l_2(T_2)$ shown below will appear on the following walks in (\ref{5'1}), (\ref{5'2}), (\ref{5'3}):
\begin{figure}[H]
\centering
\includegraphics[width=10cm]{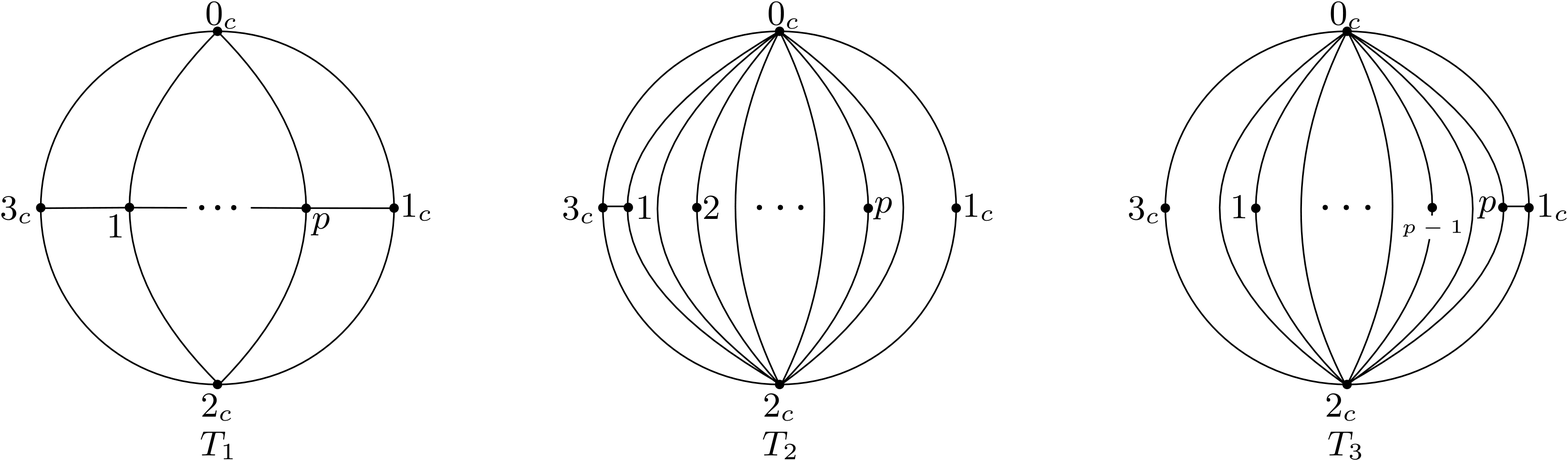}
\end{figure}
Firstly, we have
\begin{equation}\label{5'1}
l_1(T_1)=(\sum\nolimits_{i=p}^{i=2}f_{C^-,(i,i-1)})|_{p\geq2}\circ f_{\tilde{B}^-,(1_c,p)}(T_1),
\end{equation}
then $\Phi_1=\{6p+6,6p+5-2s|s=0,\cdots,p-1\}$;

Secondly, we have
\begin{equation}\label{5'2}
l_2(T_2)=f_{\tilde{E}^+,(p,1_c)}\circ f_{\tilde{E}^-,(1,3_c)}(T_2),
\end{equation}
then $\Phi_2=\{4p+4,4p+7\}$;

Thirdly, we have
\begin{equation}\label{5'3}
l_3(T_3)=(\sum\nolimits_{i=p-2}^{i=1}f_{C^+,(i,i+1)})|_{p\geq3}\circ (f_{C^+,(p-1,1_c)})|_{p\geq2}\circ f_{A^+,(p,1_c)}\circ f_{C^-,(p,0_c)}(T_3),
\end{equation}
then $\Phi_3=\{4p+5,4p+6+2s|s=0,\cdots,p-1\}$.

Thus we obtain $\Phi_1\cup\Phi_2\cup\Phi_3=[4p+4,6p+6]$.

It is easy to see that on this composite walk, there are only two triangulations $T_2$ and $T_3$ equipped with the same number of arrows which is $4p+7$. They are isotopy but are different vertices of $\mathbf{E^\circ(S,M)}$.

Hence $l_3\circ l_2\circ l_1(T_1)$ is an extended complete walk for $\mathbf{S_5}$.
\vspace{1 ex}

\textbf{Case6:} For $\mathbf{S_6}$ where $g=0,b=1,p\geq1,c\geq5$, let $T_1$ be a triangulation shown below. Then a walk staring from $T_1$ is $l_1(T_1)=(\sum\nolimits_{i=3}^{i=\frac{c+m}{2}-1}f_{H,((2i)_c,1)})|_{c\geq7}\circ \sum\nolimits_{i=1}^{i=\frac{c-m}{2}-1}f_{\tilde{B}^-,((2i+1)_c,1)}\circ f_{\tilde{B}^-,(1_c,p)}(T_1)$, so that $\Phi_1$ is $[6p+\frac{7}{2}c+\frac{m}{2}-6,6p+3c-6]\cap\mathbb{N}$.

We consider the part of the triangulation $T_2=l_1(T_1)$ shown below where punctures are located. Then the following walk $l_2(T_2)$ is to flip in this rectangle and keep the other part unchanged.
\begin{figure}[H]
\centering
\includegraphics[width=7cm]{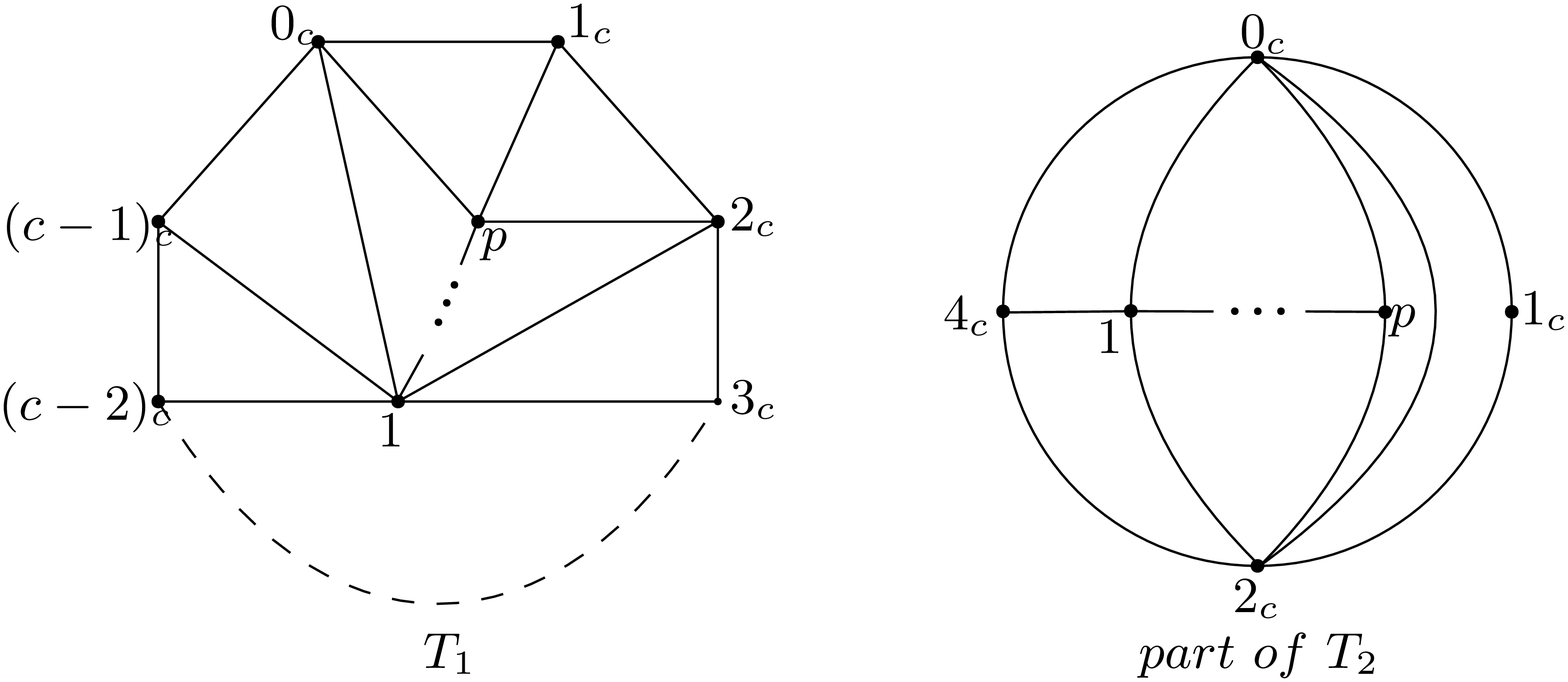}
\end{figure}
$$l_2(T_2)=(\sum\nolimits_{i=3}^{i=p}f_{C^+,(i,i-1)})|_{p\geq3}\circ (f_{C^+,(2,4_c)})|_{p\geq2}\circ f_{A^+,(1,4_c)}\circ f_{C^-,(1,2_c)}\circ(\sum\nolimits_{i=p}^{i=2}f_{C^-,(i,i-1)})|_{p\geq2}(T_2),$$
so that we have $\Phi_2=[4p+\frac{7}{2}c+\frac{m}{2}-6,6p+\frac{7}{2}c+\frac{m}{2}-6]\cap\mathbb{N}$.

Thus we obtain$\Phi_1\cup\Phi_2=[4p+\frac{7}{2}c+\frac{m}{2}-6,6p+3c-6]$. And on this composite walk, any triangulation is equipped with different number of arrows.

Hence $l_2\circ l_1(T_1)$ is an extended complete walk for $\mathbf{S_6}$.
\vspace{1 ex}

\textbf{Case7:} For $\mathbf{S_7}$ where $g=0,b\geq2,p\geq0,c\geq2$ or $g\geq1,b\geq1,p\geq0,c\geq1$, suppose $l_1'$ is the walk and $T_1'$ is the triangulation given in the Case5 in the proof of Lemma \ref{perfect1}. Then the end-point of $l_1'(T_1')$ is the triangulation equipped with $\tilde{t}_{max}$ arrows. Let $T_1=l_1'(T_1')$ and $T_2=T_1'$, then $T_2$ can be obtained from $T_1$ by $\frac{c-m}{2}$ flips of type $\tilde{B}^-$. Thus we get $l_1(T_1)$, so that $\Phi_1$ is $[2n+\frac{c+m}{2},2n+c]\cap\mathbb{N}$.

When $p=0$, an extended complete walk is $l_1(T_1)$.

When $p=1$, we consider the part of $T_2$ shown in below, that is, the part of $T_1'$, where the puncture is located. Then an extended complete walk is $f_{A^+,(p,v_1)}\circ f_{C^-,(p,v_2)}\circ l_1(T_1)$.
\begin{figure}[H]
\centering
\includegraphics[width=3cm]{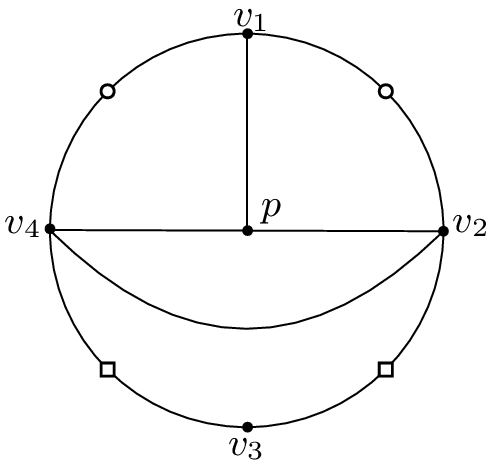}
\end{figure}

When $p\geq2$, we consider the part of the triangulation $T_2=l_1(T_1)$, where the block with $p$ punctures is located. Because the block replaces an arc which is not the common side of $\triangle_1$ and $\triangle_3$, the block lies in a rectangle of which at most one side is frozen. Thus this part of the triangulation $T_2$ is the same as the part shown in the third graph of Case5 in the proof of Lemma \ref{perfect1}. We can make the same flips in $T_2$. Let $l_2$ and $l_3$ be the walks given in the Case5 in the proof of Lemma \ref{perfect1}.

Thus an extended complete walk is $l_3\circ l_2\circ l_1(T_1)$.
\end{proof}

\begin{Theorem}
Let $\tilde{Q}$ be an extended exchange cluster quiver of finite mutation type, then the distribution set for $Mut[\tilde{Q}]$ is continuous unless $\tilde{Q}$ arises from the once-punctured triangle $\mathbf{S_{1p-tri}}$ or the forth-punctured sphere $\mathbf{S_{4p-sphere}}$.
More specifically, we have
$$\tilde{W}_{\mathbf{S_{1p-tri}}}=\{6,7,9\},\ \tilde{W}_{\mathbf{S_{4p-sphere}}}=\{8,9,10,12\}.$$
\label{2}
\end{Theorem}
\begin{proof}
  Combining Lemma \ref{ex-78}, Lemma \ref{ex-bulianxu}, Lemma \ref{perfect1} and the fact that $Mut[\tilde{Q}]=Mut[Q]$ where $\tilde{Q}$ has $|\tilde{Q}_0|=2$ and $|\tilde{Q}_1|=n$, the conclusion is complete proved.
\end{proof}

\section{The number of arrows of cluster quivers of infinite mutation type}

Lastly, we discuss the same question for cluster quivers of infinite mutation type.

\subsection{Distribution of the numbers of arrows}

\begin{Fact}
Let $Q$ be a quiver of infinite mutation type, in general, then the distribution set for $Mut[Q]$ is not continuous.
\end{Fact}
Here we present a simple example.

\begin{Example}
  Let $Q$ be the following quiver with $n>2$.
  \begin{figure}[H]
    \centering
    \includegraphics[width=30mm]{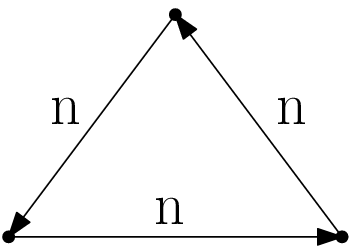}
  \end{figure}
  According to the definition of mutations, we can see that the number of arrows of any $Q^{\prime}\in$$Mut[Q]$ must be a multiple of $n$ and there are at least two quivers in $Mut[Q]$ which have different numbers of arrows. Thus this is an example of a discrete distribution set.
\end{Example}

\subsection{The number of arrows between two vertices}

First, we will show that for any quiver in $Mut[Q]$ where $Q$ is of infinite mutation type, there is no maximal number of arrows between any two vertices.
As a consequence, in this case, we have
\begin{Fact}
Let $Q$ be a quiver of infinite mutation type, then there is no upper bound on the distribution set for $Mut[Q]$.
\end{Fact}

\begin{Lemma}[\cite{FST2}] Let $Q$ be a connected quiver with more than two vertices. Then $Q$ is mutation-finite if and only if the multiplicity of the arrow between any two vertices $s,t$ is at most $2$ for any $Q^{\prime}\in$ Mut$[Q]$.
\label{mutation-finite}
\end{Lemma}

\begin{Lemma}[\cite{FST1}]
   Let $Q_{1}$ be a proper subquiver of $Q$, and $Q_{2}$ be a quiver mutation equivalent to $Q_{1}$. Then there is a quiver $Q^{\prime}$ which is mutation equivalent to $Q$ and $Q_{2}$ is a subquiver of it.
   \label{sub}
\end{Lemma}

With the help of two lemmas above, we have the following equivalent conditions for mutation-infiniteness.

\begin{Proposition}\label{>4}
  Let $Q$ be a connected quiver with $n\geqslant 3$ vertices and $N$ be any integer not less than 2. Then the following conditions are equivalent:
\begin{enumerate}[(a)]
     \item $Q$ is mutation-infinite.
     \item There is a quiver $Q^{\prime}\in$ Mut$[Q]$ satisfying that the number of the arrows between $s,t$ is larger than $2$, for some $s,t\in[1,n]$.
     \item There is a quiver $Q^{*}\in$ Mut$[Q]$ such that the number of arrows between $s,t$ is larger than $N$ for any vertices $s,t$.
\end{enumerate}
\end{Proposition}
\begin{proof}
The equivalence of (a) and (b) is exactly Lemma \ref{mutation-finite}. And (c) leading to (b) is evident, we next prove (b) leading to (c). Fix an integer $N\geqslant 2$.

\textbf{Case 1:} $n=3$. In fact this is proved in the proof of Lemma \ref{mutation-finite} in \cite{FST2}. Assume that there are three arrows of multiplicities $a,b,c$ respectively where $a\geqslant b\geqslant c$ and $a\geqslant 3$. In the rest, slightly abusing notations, we will also use the multiplicities $a,b,c$ to represent the corresponding arrows respectively. If above arrows are cyclically oriented, then mutating at the vertex not incident to arrow c makes the new arrow with multiplicity $ab-c>b$ and the new quiver is still cyclically oriented. Iteratively mutate at vertex not incident to the arrow of smallest multiplicity, consequently all three arrows have multiplicities larger than $N$. If they are not cyclically oriented, then there are two arrows in the same orientation. Mutating at the vertex incident to them leads to a cyclically oriented quiver.

\textbf{Case 2:} $n=4$. Because of the $n=3$ case and Lemma \ref{sub}, every connected quiver with an arrow having multiplicity more than 2 is mutation equivalent to one of following form:
\begin{figure}[H]
  \centering
  \includegraphics[width=2cm]{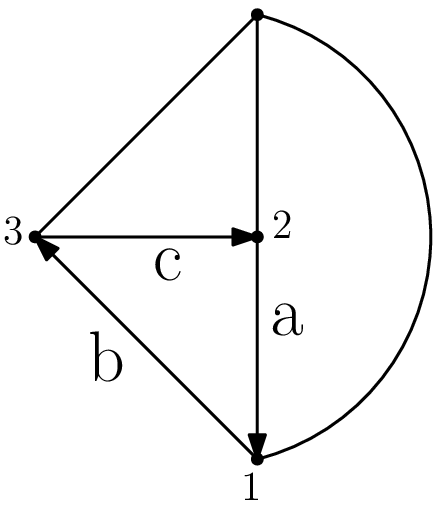}
\end{figure}
where $a>b>c>N$. The orientations and multiplicities of other arrows are not sure. (multiplicity being $0$ means the absence of this arrow.) But the quiver is still connected since connection is mutation-invariant. It is showed in Appendix that iteratively applying mutations $\mu_{3}\mu_{2}\mu_{1}$ always change the quiver to $Q^{*}$ (see Figure \ref{Q*}).
\begin{figure}[H]
  \centering
  \includegraphics[width=2cm]{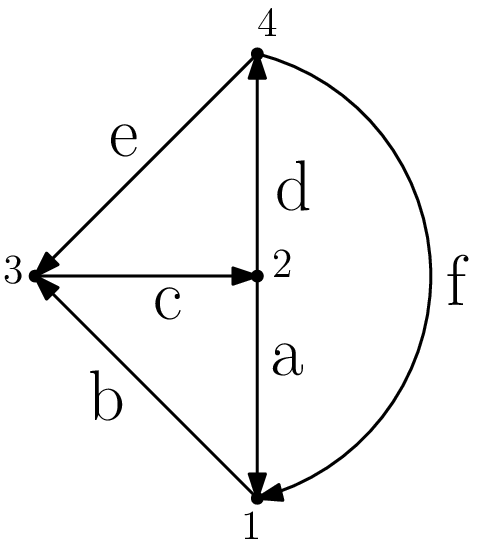}
  \caption{$Q^{*}$}
  \label{Q*}
\end{figure}
With little abuse of notations, here we still use $a,b,c,d,e,f$ to represent arrows of quivers after mutations. Similar to the discussion in case $n=3$, in the process presented in Appendix, we can see that $\mu_{i}$ ($i=1,2,3$) always keeps the multiplicity of arrow incident to $i$ and changes the other from the least to the largest among $a,b,c$. Therefore in $Q^{*}$, we still have $a>b>c>N$. Besides, if continue to apply $\mu_{3}\mu_{2}\mu_{1}$, the quiver goes through a circulation with all multiplicities growing as the following figure.
\begin{figure}[H]
  \centering
  \includegraphics[width=8cm]{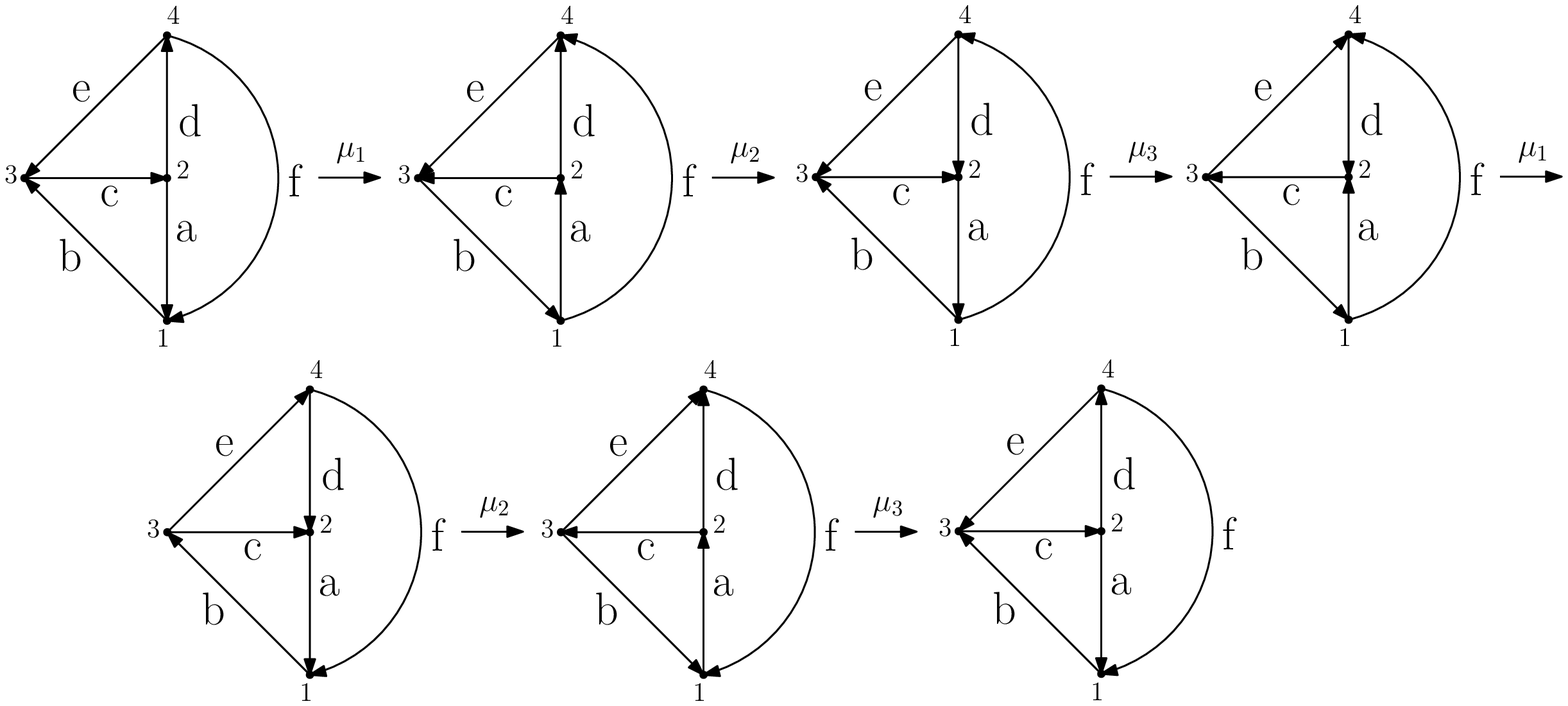}
\end{figure}
Thus after enough (even) times $\mu_{3}\mu_{2}\mu_{1}$, we get a quiver with the same arrows satisfying $a>b>c$, $d>e>c$, and $c,f>N$.

\textbf{Case 3:} $n>4$. Denote by $\Delta_{123}$  the full subquiver with vertices 1, 2 and 3. Similar things happen to vertices connected to the triangle $\Delta_{123}$ in original quiver, which means that by iteratively applying mutations $\mu_{3}\mu_{2}\mu_{1}$, the subquiver with vertices 1, 2, 3 and any vertex connected to the triangle $\Delta_{123}$ (we label this vertex 4) is of form $Q^{*}$ with $a>b>c$, $d>e>c$, and $c,f>N$. Next if we iteratively apply mutations $\mu_{3}\mu_{2}\mu_{4}$ instead of $\mu_{3}\mu_{2}\mu_{1}$, similarly as case 2, we can keep the multiplicities of arrows incident to a vertex of the triangle $\Delta_{123}$ larger than $N$ with direction unchanged, but additionally make the full subquiver consisting of vertices 4, 2, 3 and a vertex either connected to the triangle $\Delta_{423}$ or that connected to the triangle $\Delta_{123}$ to be of form $Q^{*}$ with $a>b>c$, $d>e>c$, and $c,f>N$, as shown in the following figure.
\begin{figure}[H]
  \centering
  \includegraphics[width=14cm]{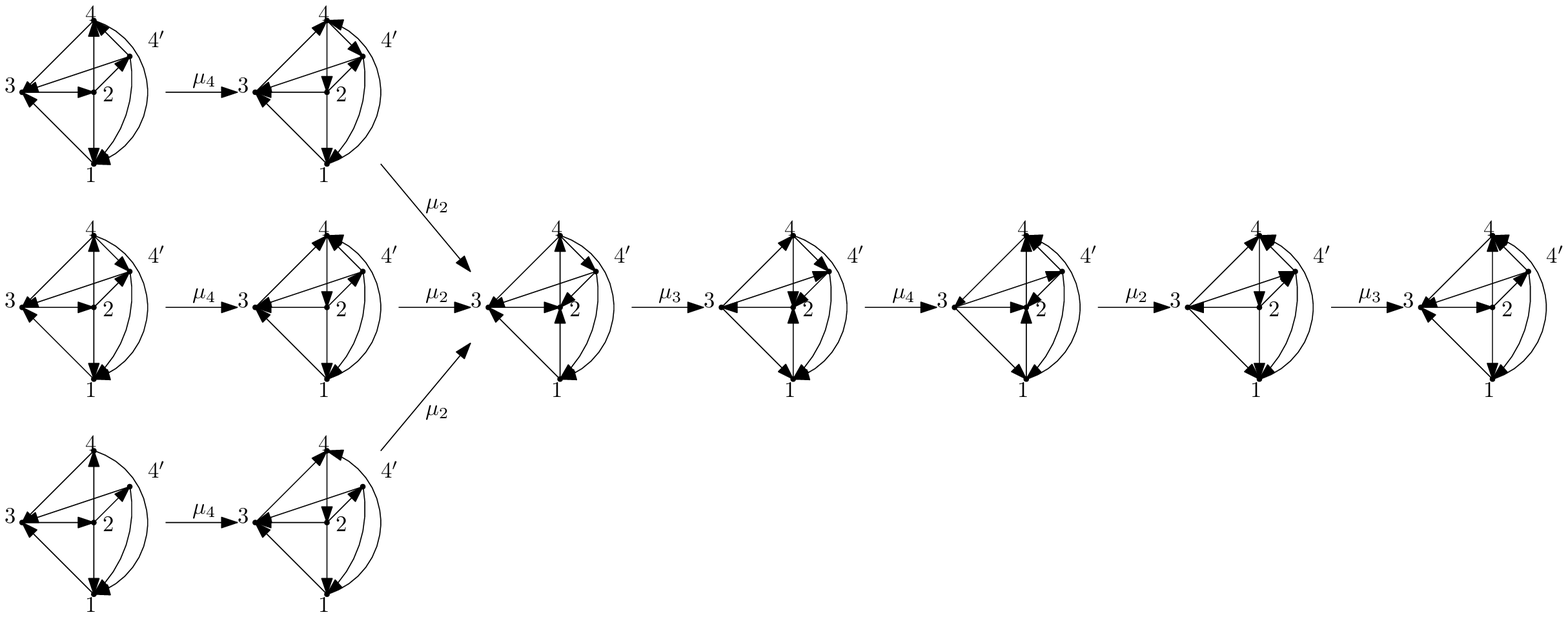}
\end{figure}
Therefore, because of the finiteness of the number of vertices in a quiver, we can obtain a quiver in $Mut[Q]$ whose the number of multiplicities of arrows between any two vertices may be larger than any $N$ by the following procedures:
\begin{enumerate}
  \item Label all vertices by $1,2,3,\ldots,n$, particularly let subquiver incident to 1, 2 and 3 be connected;
  \item Iteratively apply mutations $\mu_{3}\mu_{2}\mu_{1}$ until all arrows connecting vertices and a vertex of triangle $\Delta_{123}$ having multiplicity larger than $N$;
  \item Choose a mutation sequence $\mu_{i}\mu_{j}\mu_{k}$ other than those we have applied in the former steps, where the multiplicities of arrows in $\Delta_{ijk}$ is larger than $N$ due to above steps and $m_{i}>m_{j}>m_{k}$ (here $m_{i}$ denotes the multiplicity of the arrow in $\Delta_{ijk}$ not incident to $i$). Iteratively apply $\mu_{i}\mu_{j}\mu_{k}$ until all arrows incident to a vertex of triangle $\Delta_{ijk}$ having multiplicity larger than $N$;
  \item Repeat the former step enough times to make all multiplicities of arrows larger than $N$.
\end{enumerate}
\end{proof}

As a corollary, the above result can be restated in the language more related to former sections as follows.
\begin{Corollary}
  Let $Q$ be a connected quiver with $n\geqslant 3$ vertices, then the following are equivalent:
  \begin{itemize}
    \item $Q$ is mutation-infinite.
    \item There is no upper bound on the distribution set for $Mut[Q]$.
    \item There is a quiver in $Mut[Q]$ having lager than 2 arrows between certain two vertices.
    \item For any two vertices of $Q$, the maximal number of arrows between this two vertices of a quiver in $Mut[Q]$ is infinite.
  \end{itemize}
\end{Corollary}

\begin{Remark}
  Although we only focus on mutation-infinite cluster quivers corresponding to mutation-infinite skew-symmetric cluster algebras, the whole method still works for mutation-infinite skew-symmetrizable case by replacing quivers with weighted quivers and replacing multiplicities of arrows with weights. Then with almost the same discussion, we can get the analogous results in the skew-symmetrizable case.
\end{Remark}

\section{Appendix}
In this appendix we show that when $n=4$, iteratively applying mutations $\mu_{3}\mu_{2}\mu_{1}$ always changes the quiver to $Q^{*}$. And during this process the multiplicities of $a,b,c$ never decrease.
\begin{figure}[H]
  \centering
  \includegraphics[width=16cm]{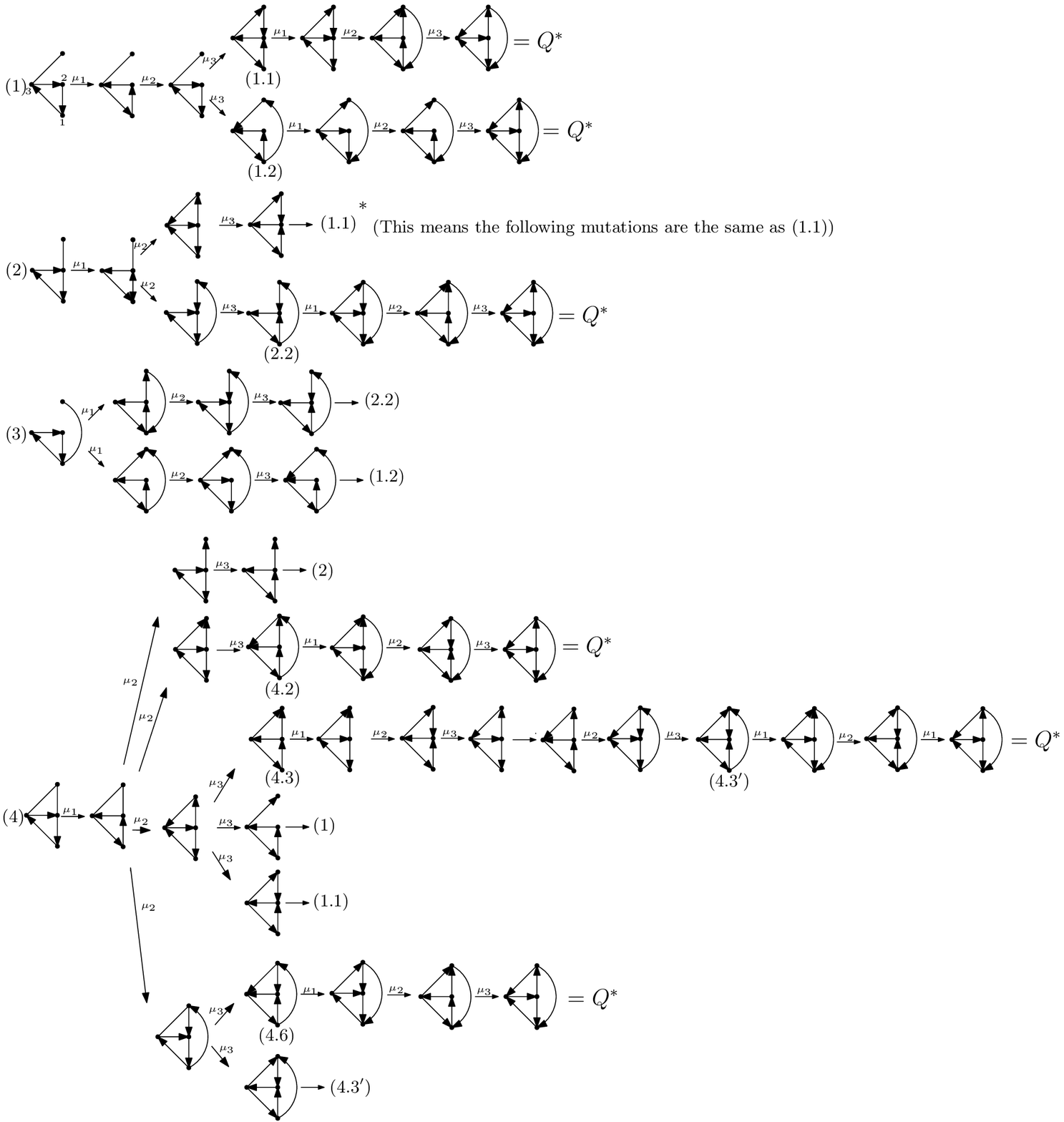}
\end{figure}
\begin{figure}[H]
  \centering
  \includegraphics[width=16.5cm]{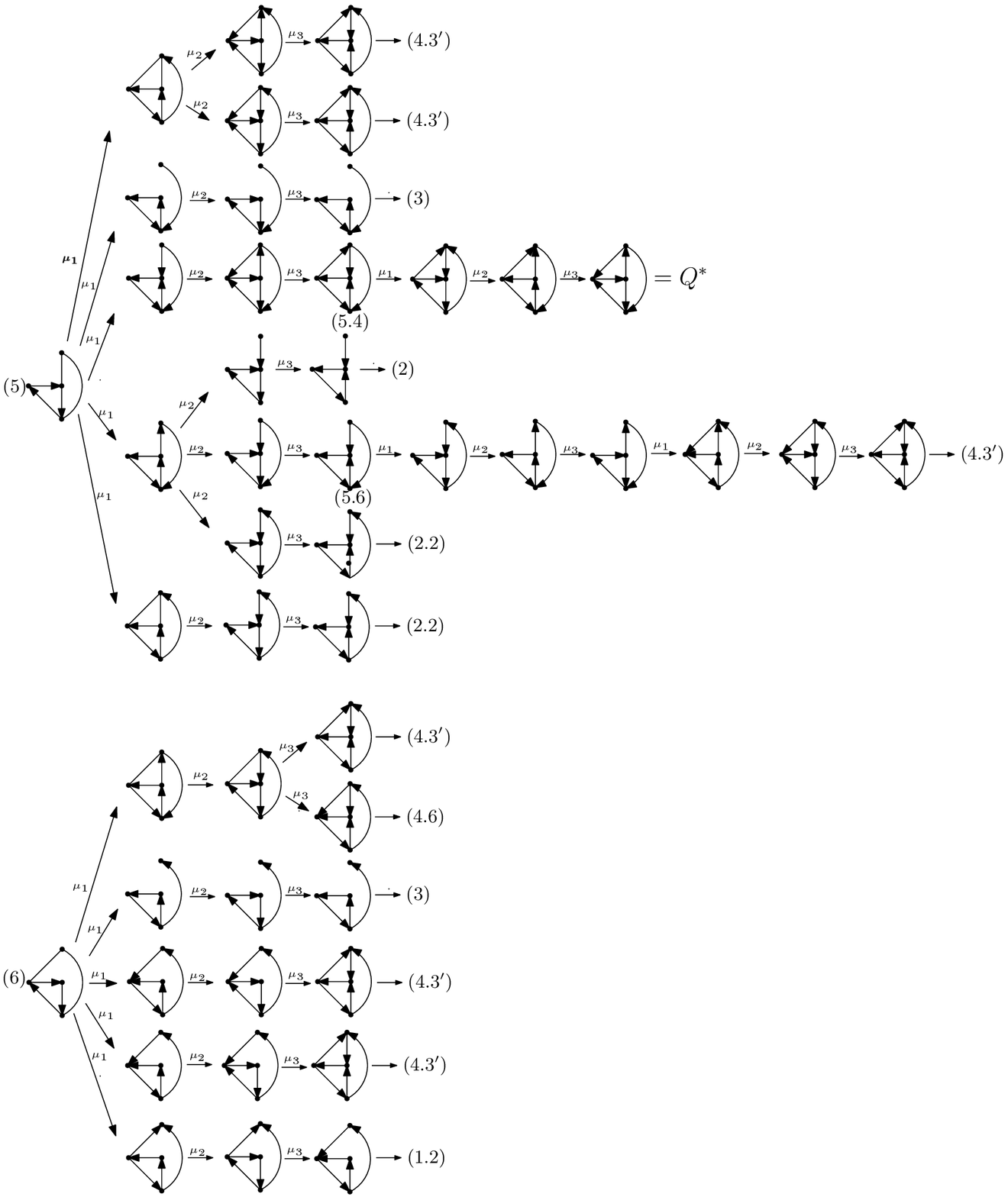}
\end{figure}
\begin{figure}[H]
  \centering
  \includegraphics[width=14cm]{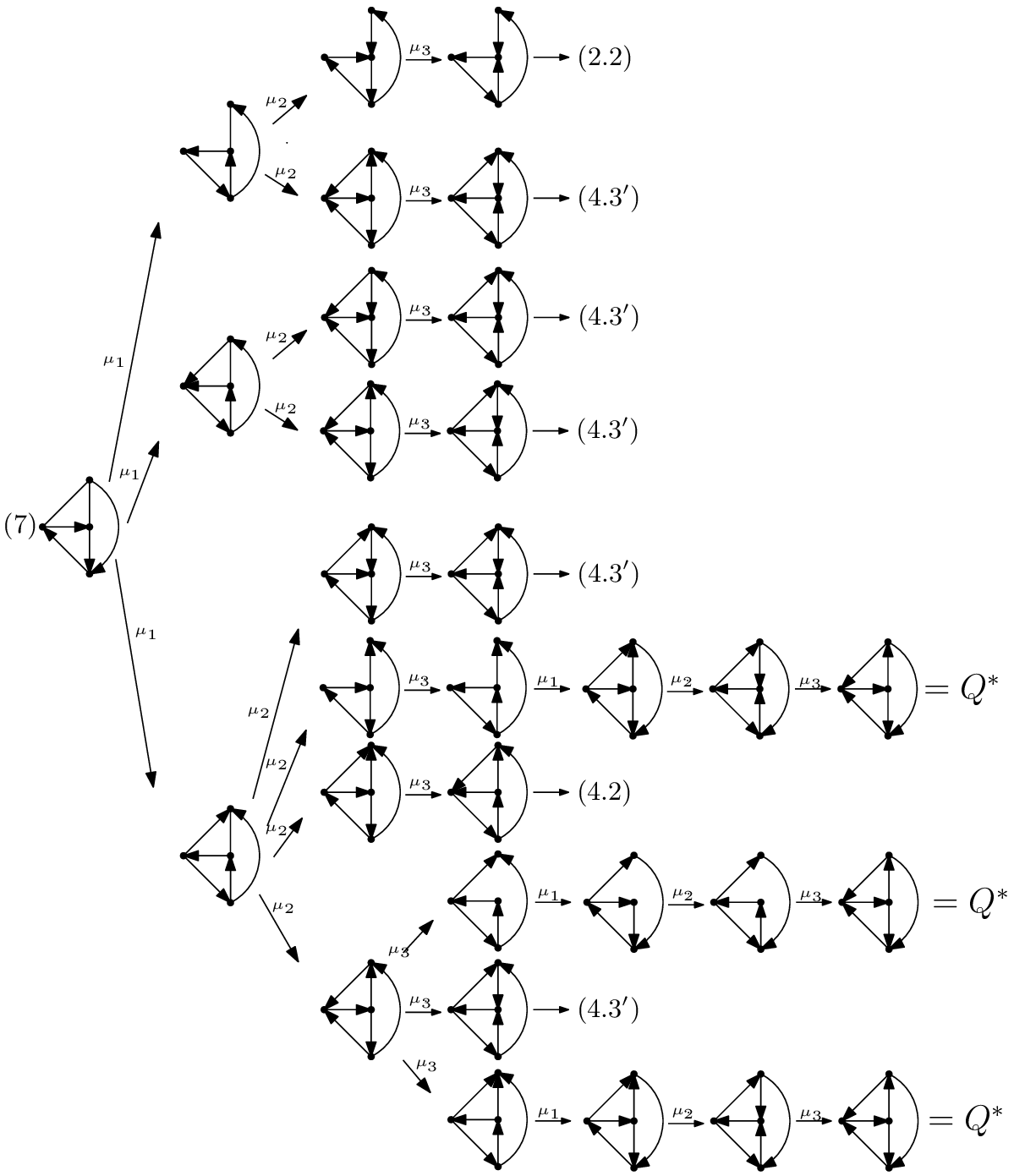}
\end{figure}
\begin{figure}[H]
  \centering
  \includegraphics[width=14cm]{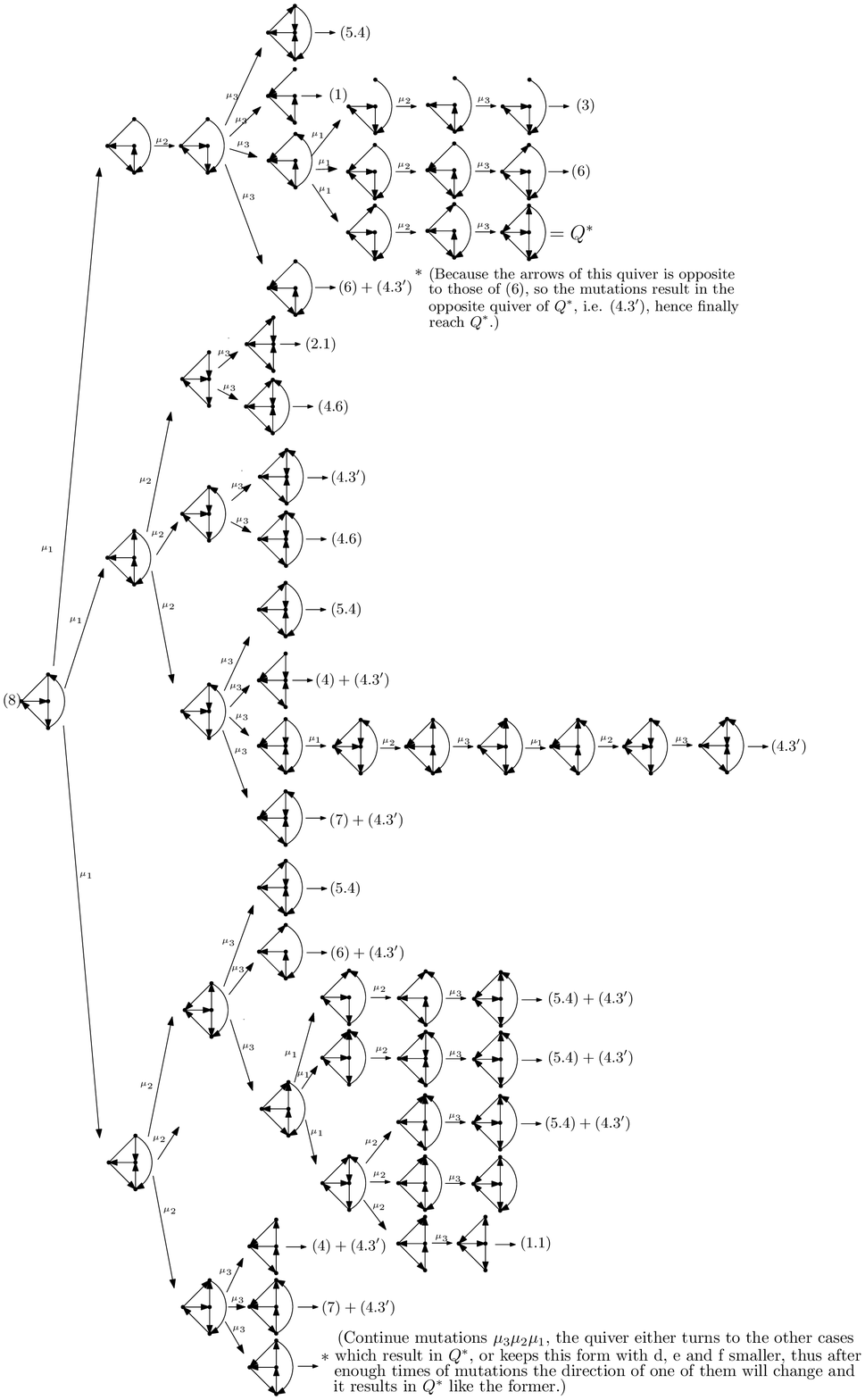}
\end{figure}
\vspace{3mm}

{\bf Acknowledgements}: {\em This project is supported by the National Natural Science Foundation of
China (No.12071422) and the Zhejiang Provincial Natural Science Foundation of China (No.LY19A010023)}.

\bibliographystyle{amsplain}

\end{document}